\documentclass[10pt]{article}
\usepackage[english]{babel}
\usepackage[T1]{fontenc}
\usepackage[utf8]{inputenc}

\usepackage{graphics}	
\usepackage{wrapfig}
\usepackage{graphicx} 	
\usepackage{subfig} 	

\usepackage{rotating} 	

\usepackage{enumitem}  	

\usepackage{mathrsfs}  	

\usepackage{amsmath}
\usepackage{amsthm}
\usepackage{amsfonts}

\usepackage{fancyhdr}

\usepackage{url}

\usepackage[colorinlistoftodos]{todonotes}

\usepackage{pdfpages}

\usepackage{hyperref, breakcites}

\usepackage[export]{adjustbox}	

\usepackage{framed}

\theoremstyle{remark}

\newtheorem{Rem}{Remark}

\theoremstyle{definition}

\newtheorem{Thm}{Theorem}
\newtheorem{Cond}{Conditions}

\newtheorem{Def}{Definition}
\newtheorem{Cor}{Corollary}
\newtheorem{Prop}{Proposition}

\usepackage{authblk}

\author[1]{Christian Mandler\thanks{Corresponding author. \\ \textit{E-mail addresses: \href{mailto:christian-mandler@math.uni-giessen.de}{Christian.Mandler@math.uni-giessen.de},  \href{mailto:Ludger.Overbeck@math.uni-giessen.de}{Ludger.Overbeck@math.uni-giessen.de}}}}
\author[1]{Ludger Overbeck}
\affil[1]{Mathematisches Institut, Justus-Liebig-Universit\"{a}t Gießen, Gießen, Germany}

\title{\textbf{A Functional It\={o}-Formula for Dawson-Watanabe Superprocesses}}

\begin{document}

\maketitle 
\thispagestyle{empty}

%



\begin{abstract}
\noindent
We derive an It\={o}-formula for the Dawson-Watanabe superprocess, a well-known class of measure-valued processes, extending the classical It\={o}-formula with respect to two aspects. Firstly, we extend the state-space of the underlying process $(X(t))_{t\in [0,T]}$ to an infinite-dimensional one - the space of finite measure. Secondly, we extend the formula to functions $F(t,X_t)$ depending on the entire paths $X_t=(X(s\wedge t))_{s \in [0,T]}$ up to times $t$. This later extension is usually called functional It\={o}-formula. Finally we remark on the application to predictable representation for martingales associated with superprocesses.  \\

\textbf{MSC2020 subject classifications:} 60J68; 60H05; 60G07; 60G57
\noindent

\end{abstract}

\section{Introduction}

Dawson-Watanabe Superprocesses are measure-valued Markov processes and as such take values in infinite dimensional spaces. Therefore, a fundamental tool in stochastic calculus, the traditional It\={o}-formula, is not directly applicable to functions of such processes. Dawson \cite{Daw78} proved an It\={o}-formula for measure-valued process but his result is limited to what we call \emph{finitely based functions} and compares to Theorem \ref{Thmfb} in this work. This result is related to the martingale problem associated to -  or even defining -  a superprocesses, whose generator is usually only derived for these finitely-based functions. In \cite{JT03} the domain of generator was extended to a wider class of functions. In the present work, we give, in a first step, the It\={o}-formula for functions in this extended domain. The main result (Theorem \ref{Thm.fct.Ito}), however, is the  \emph{functional} extension of this It\={o}-formula for superprocesses.  \emph{Functional} means that, instead of functions of the current state $X(t)$ of the superprocess, we can consider functions $F(t, X_t)$ of the paths $X_t=(X(t \wedge s))_{s\in [0,T]}$ up to time $t$.  \\

In stochastic calculus, Dawson-Watanabe Superprocesses have attracted a lot of interest since their introduction by Watanabe \cite{Wat68} and Dawson \cite{Daw75}. As one of the well understood non-Gaussian infinite-dimensional diffusion processes they have been studied for several reasons. For one, they arise naturally as a scaling limit of  branching particle systems (see Chapter 6 in \cite{Daw93} or \cite{Dy91}, \cite{DuPe99} and \cite{DuMyPe05}). This approximation enables one to derive many results for the limiting process. Also, they are closely related to a certain class of non-linear partial differential equations (see \cite{Dy93}, \cite{Dy02} or \cite{LeG99}) which provides a fruitfull interplay between non-linear PDEs and stochastic analysis. Finally, they are related to a second class of (probability-)measure-valued diffusions - the Fleming-Viot process (\cite{EthMa91}), which has applications in population genetics (as Dawson-Watanabe-superprocesses are related to population \emph{growth}). More recently, because of their relation to non-linear PDEs, they also have been helpful in some problems in mathematical finance as in \cite{Sch13} or \cite{GHL14}. For thorough introductions to the topic, we refer to the notes of Dawson \cite{Daw93}  and Perkins \cite{Per99} and  for an introductory reading to \cite{Eth00} and \cite{Daw17}. \\

In this work, we focus on a special class of Dawson-Watanabe superprocesses, the so-called $B(A,c)$-superprocesses, which includes the super-Brownian motion. These processes are obtain by considering the scaling limit of a branching particle systems on a locally compact metric space $E$ with the particle motion described by the generator $A$, branching rate $dt$ and the branching mechanism described by $\Phi(\lambda) = -\frac{1}{2} c \lambda ^2$. A $B(A,c)$-superprocess can best be characterized by the martingale problem given in Section \ref{SecSP}. Basically, we take advantage of two useful properties of $B(A,c)$-superprocesses. First, the characterizing martingale problem yields a continuous orthogonal $L^2$-martingale measure (Example 7.1.3 in \cite{Daw93}) -- a concept introduced by Walsh \cite{Wal86} that plays a substantial role in our two main results\footnote{We provide a short introduction to the topic in Section \ref{SecMM}.}. The second  property is the almost sure continuity of the sample paths of $B(A,c)$-superprocesses (Theorem 4.7.2 in \cite{Daw93}) which we use  for the result in Section \ref{SecHP}. For superprocesses sharing these two properties one can prove similar results as in the present paper. However, as for $\mathbb{R}^d$-valued processes, an extension of the functional It\={o}-formula to the discontinuous case is also feasible. \\

Our first result (Theorem \ref{Thm.Ito}) is an It\={o}-formula for functions of $B(A,c)$-super-processes. This result can be viewed as an extension to the It\={o}-formula presented in \cite{Daw78} applied to the $B(A,c)$-superprocess. We obtain the result Section \ref{SecSP} by reformulating a result in \cite{JT03} using the above mentioned martingale measure derived from the martingale problem. \\

In Section \ref{SecHP} we prove the functional version (Theorem \ref{Thm.fct.Ito}) of the It\={o}-formula in Theorem \ref{Thm.Ito}. While the traditional It\={o}-formula only takes into account the current state $X(\cdot)$ of a stochastic process $X$, a functional It\={o}-formula considers the whole history of the process, i.e. $X(t \wedge \cdot)$, the path of $X$ up to time $t$. This notion goes back to the landmark paper \cite{Du09} by Dupire. Picking up Dupire's approach, Cont and Fournié \cite{CoFo13} formalized some of the concepts presented in \cite{Du09}. In the companion paper \cite{CoFo10}, the authors extend the functional It\={o}-formula to processes with cadlag paths. Using an alternative approach, Levental et al. also proved a version of the functional It\={o}-formula for cadlag semimartingales in \cite{Lev13}. As we work with the alternative approach based on \cite{Co16}, we devote part of Section \ref{SecApp} to a comparison of this approach to the \emph{vector bundle} approach found in \cite{Du09}, \cite{CoFo10} and \cite{CoFo13}. \\

In the literature referred to above, only $\mathbb{R}^d$-valued processes were considered. Our result is an extension of Dupire's result to certain measure-valued processes, i.e. processes taking values in an infinite dimensional space.  Another extension of the function It\={o}-formula to infinite dimensional spaces is due to Rosestolato, who proved such a formula for Hilbert-space-valued processes in \cite{Ro16}. \\

A completely different approach to the topic is due to Russo and various co-authors. They looked at the topic from the standpoint of processes taking values in separable Banach spaces (which include the path space $C([-T,0])$) and derived a stochastic calculus for such processes. In \cite{CoRu18} the authors reformulate the functional It\={o}-formula using this approach and compare their approach to the one by Cont and Fournié. For more details on this approach, we refer to the references in \cite{CoRu18}. \\

One application to the functional It\={o}-formula presented in \cite{Du09} as well as \cite{CoFo13} and \cite{Ro16} is the derivation of an explicit martingale representation formula. In its most basic form the martingale representation formula goes back to It\={o} who proved that every square-integrable functional of a Brownian motion can be expressed as a constant plus an integral against the Brownian motion. This result has since been extended by Clark, Hausmann and Ocone, leading to the famous Clark-Haussmann-Ocone formula (\cite{Cla70}, \cite{Hau79},\cite{Oco84}). In their approach, the authors proceeded in two steps. First, some differentation on the path space takes place (like Malliavian- or H-derivative) and then one has to compute the predictable projection of the derivative.  As mention in \cite{CoFo13}, the representation obtained from the functional It\={o}-formula can be seen as a \emph{nonanticipative counterpart of the Clark-Haussmann-Ocone formula} and yields an alternative to the traditional approach of using Malliavin calculus to compute the integrand in the martingale representation. \\

Martingale representation formulas for functionals of superprocesses have already been the subject of interest in \cite{EP94}, \cite{EP95} and  \cite{Ov95}. The later analysed  the equivalence of the existence of a predictable representation of a martingale with respect to a filtration associated with a  martingale measure and the extremality of the underlying probability measure for general measure-valued processes. \\

In \cite{EP94}, the authors proved the existence of a representation for certain functionals $F$ of the superprocess as an integral with respect to the orthogonal martingale measure. In the follow-up paper \cite{EP95}, they  derived the explicit form of the integrand by making use of historical processes, an enriched version of superprocesses that keeps track of genealogies (see \cite{EP95} or \cite{DawPer91}). Their approach is in the spirit of a variational approach using Girsanov transformation as in \cite{Bis81} and yields a \emph{stochastic integration-by-parts formula}.\\

In Section \ref{SecApp} we derive the explicit form of the martingale representation formula using Theorem \ref{Thm.fct.Ito}, i.e. a completely different technique than Evans and Perkins. Just as in \cite{CoFo13}, our explicit formula can be seen as counterpart to the result in \cite{EP95}. A detailed, more thorough study of the comparison to the results in \cite{EP95} is subject of ongoing work. An analysis of the weak martingale representation is the subject of \cite{MO20}.

\section{An It\={o}-Formula for the Dawson-Watanabe Superprocess} \label{SecSP}

Let us first introduce the necessary notations. Consider a compact\footnote{The results can be extended to locally compact spaces by considering the one point compactification of $E$.} metric space $E$ with its Borel-$\sigma$-algebra $\mathcal{E}$. $E$ is the state space of the so called \emph{one particle motion}, which is defined by its infinitesimal generator $A$ generating a strongly continuous Markov semigroup $\{P_t\}_{t \geq 0 }$ acting on $C(E)$, the space of continuous functions on $E$ equipped with the supremum norm. We denote the domain of $A$ by $D(A)$ and assume that there is a dense linear subspace $D_0$ of $C(E)$ that is an algebra and closed under the mapping $P_t$ for all $t \geq 0$. If such a subspace exists, $A$ is called a good generator in  \cite{JT03}\footnote{E.g.\ in the case of the super-Brownian motion (i.e.\ $A$ is the Laplace operator), the algebra $D_0$ is given by the algebra generated by the Schwartz space of rapidly decaying test functions and the constant function 1 (see \cite{JT03}).}. We sometimes write $A^{(x)}$ when we want to emphasize that the generator is applied to functions of the variable $x$. \\

As mentioned above, the $B(A,c)$-superprocess is the diffusion limit of a branching particle system in which the particles move according to $A$ and split after exponential life times into a critical number of new particles. The branching particle system takes values in the set of finite point measures on $E$ and its diffusion limit can be defined as a Markov process with state space $M_F(E)$, the space of finite measures on $(E, \mathcal{E})$ equipped with the topology of weak convergence. We write $\langle \mu, f \rangle$ for $\int_E f d\mu$ with $f : E \rightarrow \mathbb{R}$, $\mu \in M_F(E)$ and fix a $c > 0$. \\

Fix a time $T \in [0,\infty)$, let $\Omega = C([0,T], M_F(E))$ be the space of continuous functions from $[0,T]$ to $M_F(E)$ and $X=(X(t))_{t\in [0,T]}$ be the coordinate process $X(t)(\omega)=\omega(t)$. Further, let $\mathcal{F}$ be the corresponding Borel $\sigma$-algebra induced by the sup-norm and let $(\mathcal{F}_t)_{t \in [0,T]}$ be the canonical filtration satisfying the usual conditions. Then a measure $\mathbb{P}_m$ on $\Omega$ is a solution to the martingale problem for the $B(A,c)$-superprocess started at $m \in M_F(E)$ if
\begin{equation}
\begin{gathered}
\mathbb{P}_m(X(0) = m) = 1 \text{ and for all $\phi \in D(A)$ the process}  \\
M(t)(\phi) = \langle X(t) , \phi \rangle - \langle X(0), \phi \rangle - \int_0^t \langle X(s) , A  \phi \rangle ds  , \quad t \in [0,T] \\
\text{is a $(\mathcal{F}_t)_{t\in [0,T]}$ local martingale with respect to $\mathbb{P}_m$ }  \\
\text{and has quadratic variation } [M(\phi)]_t = \int_0^t \langle X(s) , c  \phi^2 \rangle ds. 
\end{gathered}
\tag{MP} \label{MP}
\end{equation}

In this case, $\mathbb{P}_m$ is the distribution of the $B(A,c)$-superprocess. If a measure-valued stochastic process $X$  defined on a stochastic basis $(\Omega, {\cal F},\mathbb{ P}, {\cal F}_t)$ has a distribution solving \eqref{MP}, we also call $X$ a solution of the martingale problem. Note that the solution to \eqref{MP} is unique, that the resulting local martingale $M(t)(\phi)$ is a true martingale and that the resulting process $X$ has continuous paths, cf. \cite{Daw93}. \\

Next, for continuous functions  $F: [0,T] \times M_F(E) \rightarrow \mathbb{R}$ define the directional derivative in direction $\delta_x$ for $x \in E$ by 
\begin{equation*}
D_x F(t, \mu) = \lim_{\varepsilon \rightarrow 0} \frac{F(t, \mu + \varepsilon \delta_x) - F(t, \mu)}{\varepsilon} ,
\end{equation*}
if the limit exists. Higher order directional derivatives are defined by iteration. Set $D_{xy}F(s,\mu) = D_x D_y F(s,\mu)$. If the derivatives are continuous, we have $D_{xy}F(s,\mu) = D_{yx}F(s,\mu)$. When considering derivatives with respect to the time, write $D_sF(s, \mu)$ and if mixed derivatives are considered $D_{sx}F(s,\mu)$ and so on. If we denote by $DF(s,\mu)$ the function $x \mapsto D_xF(s,\mu)$, then $A^{(x)}D_xF(s, \mu)$ should be understood as $ADF(s,\mu)(x)$, i.e.\ the operator $A$ applied to the function $x \mapsto D_xF(s, \mu)$. \\

\subsection{Martingale Measures} \label{SecMM}

In order to formulate an It\={o}-formula incorporating the martingale part in the semimartingale decomposition of semimartingales $Y$ of type $Y_t = F(t, X(t))$, we have to recall the concept of martingale measures, a kind of  measure-valued martingales which goes back to Walsh (\cite{Wal86}). In the first part of this subsection, we summarize the parts in \cite{Wal86} that play a role in this work but restrict the summary to the relevant case of orthogonal martingale measures. For more general results and proofs, we refer to the original work. In the second part of this subsection, we extend the class of valid integrands of stochastic integrals with respect to martingale measures to include bounded optional processes. This is done by extending standard arguments used in the theory of stochastic integrals with respect to martingales.\\

Let $L^2 = L^2(\Omega, \mathcal{F}, \mathbb{P})$ for a probability space $(\Omega, \mathcal{F}, \mathbb{P})$ and let $(\mathcal{F}_t)_{t \in \mathbb{R}_+}$ be a right continuous filtration on this space.

\begin{Def}[Martingale Measure] \
Let $\mathcal{A}\subset \mathcal{E}$ be an algebra. 
 Then
\begin{enumerate}
\item A process $M_t(A)_{t \geq 0, A \in \mathcal{A}}$ is called a martingale measure with respect to $\mathcal{F}_t$ if
\begin{enumerate}[label = (\roman*)]
\item $M_0(A) = 0$ for all $A \in \mathcal{A}$,
\item if $t > 0$, then $M_t(\cdot)$ is a $\sigma$-finite $L^2$-valued measure\\
 (in the sense of \cite{Wal86}; in particular $\mathbb{E}[M^2_t(A)] < \infty$ and $\sup\{\mathbb{E}[M^2_t(A)]|A\in {\cal E}_n\}<\infty$, where ${\cal E}_n$ is the restriction of $\cal A$ to a subset $E_n\subset E_{n+1}, \cup E_n=E$.) 
\item $(M_t(A))_{t \geq 0}$ is a martingale with respect to the filtration $\mathcal{F}_t$.
\end{enumerate}
\item A martingale measure $M$ is called orthogonal if $A \cap B = \emptyset$ implies that $M_t(A)$ and $M_t(B)$ are orthogonal martingales, i.e. $M_t(A)M_t(B)$ is a martingale. 
\item The martingale measure $M$ is called continuous if for all $A \in \mathcal{A}$ the map $t \mapsto M_t(A)$ is continuous.
\end{enumerate}
\end{Def}

\begin{Def}[Elementary and Simple Functions] A function $f(s, x, \omega)$ is called elementary if it is of form
\begin{equation*}
f(\omega ,s,x) = X(\omega) 1_{(a , b]}(s) 1_A(x)
\end{equation*}
with $0 \leq a < b \leq T$, bounded and $\mathcal{F}_a$-measurable $X$ and $A \in \mathcal{E}$. If $f$ is a finite sum of elementary functions, it is called simple. We denote the class of simple functions by $\mathcal{S}$.
\end{Def}

\begin{Def}[Predictable Functions]
The predictable $\sigma$-field $\mathcal{P}$ on $\Omega \times L \times \mathbb{R}_+$ is the $\sigma$-algebra generated by $\mathcal{S}$. A function is called predictable if it is $\mathcal{P}$-measurable. 
\end{Def}

Now let $M$ be an orthogonal martingale measure and let $Q$ be the covariance functional  defined by 
\begin{equation*} 
Q([0,t], A, B)=[M(A),M(B)]_t.
\end{equation*}
Also, one can define a martingale $M(\phi)$ for each measurable bounded $\phi :[0,T]\times E$. Since $M$ is orthogonal, we have that $Q$ is concentrated on the diagonal and there is a measure $\nu$ with $\nu([0,t] , A) = Q( [0,t] , A , A )$. In this case the covariation $ [M(\phi),M(\psi)]$ of the martingales $M(\phi),M(\psi)$ equals
\begin{equation*}
[M(\phi),M(\psi)]_t = \int_{\mathbb{R}_+}  \int_E  \phi(s,x) \psi(s,x) \nu(ds,dx).
\end{equation*}
Using this, we can define a norm $\| \cdot \|_M$ on the predictable functions by
\begin{equation*}
\| f \|_M = \mathbb{E} \left[\int_0^T f^2(s,x)\nu(ds,dx) \right]^{\frac{1}{2}} 
\end{equation*}
which allows us to define the class $\mathcal{P}_M$ of all predictable $f$ for which $\| f \|_M < \infty$ holds. 

\begin{Prop} \label{Prop.Wal.1}
The class $\mathcal{S}$ is dense in $\mathcal{P}_M$ with respect to the norm $\| \cdot \| _M$. 
\end{Prop}

This allows us to construct a stochastic integral for functions in $\mathcal{P}_M$ with respect to an orthogonal martingale measure as follows: We start with an elementary function $f$, for which we set
\begin{equation*}
f \bullet M_t(B)  (\omega)= X(\omega) (M_{t \wedge b}(A \cap B) (\omega) - M_{t \wedge a} (A \cap B)) (\omega)
\end{equation*}
for all $0 \leq s < t \leq t$ and $A \in \mathcal{E}$ and extend the definition to $\mathcal{S}$ by linearity. \\

\begin{Prop}\label{Prop.Wal.2}
It holds for $f \in \mathcal{S}$:
\begin{enumerate}[label = (\roman*)]
\item $f \bullet M$ is an orthogonal martingale measure.
\item $\mathbb{E}[ (f \bullet M_t(B))^2 ] \leq \| f \|_M^2$ for all $B \in \mathcal{E}$ and $t \leq T$.
\end{enumerate}
\end{Prop}

For general $f \in \mathcal{P}_M$, we get from Proposition \ref{Prop.Wal.1} the existence of a series $(f_n)_n \in \mathcal{S}$ such that $\| f^n - f\|_M \rightarrow 0$ as $n \rightarrow \infty$. From Proposition \ref{Prop.Wal.2} (ii), we further get
\begin{equation*}
\mathbb{E}\left[ (f^m \bullet M_t(B) - f^n \bullet M_t(B) )^2 \right] \leq \| f^m - f^n \|_M^2 \rightarrow 0 \quad \quad \text{as $m,n \rightarrow \infty$}.
\end{equation*}
Thus, $(f^n \bullet M_t(B))_n$ is Cauchy in $L^2(\Omega, \mathcal{F}, \mathbb{P})$ and therefore converges in $L^2$ to a martingale which we call $f \bullet M_t(B)$. 

\begin{Rem}
\begin{enumerate}[label = (\roman*)]
\item The results in Proposition \ref{Prop.Wal.2} also hold for $f \bullet M$ with $f \in \mathcal{P}_M$.
\item Instead of $f \bullet M_t(B)$, from now on, we use the following notation:
\begin{equation*}
f \bullet M_t(B) = \int_0^t \int_B f(s,x) M(ds, dx) .
\end{equation*}
\end{enumerate}
\end{Rem}

\textbf{Martingale Measure for the Dawson-Watanabe Superprocess} \\

From Example 7.1.3 in \cite{Daw93} we know that the solution $M$ of the martingale problem \eqref{MP} satisfies 
\begin{equation*}
[M(\phi_1), M(\phi_2)]_t = \int_0^t \langle X(s), c \phi_1 \phi_2 \rangle ds
\end{equation*}
for all $\phi_1 , \phi_2 \in D(A)$. Further, we know that, since $D(A)$ is dense in $C(E)$ and thus $bp$-dense in the space of bounded $\mathcal{E}$-measurable functions $b\mathcal{E}$, we can extend this martingale to all functions in $b \mathcal{E}$. This finally yields a continuous orthogonal martingale measure $M(ds, dx)$ with the above mentioned measure $\nu$ given by 
\begin{equation}
\nu((0,t], B) = c \int_0^t X(s)(B) ds
\label{eq.nu}
\end{equation}
for all $t \in (0,T]$, $B \in \mathcal{E}$. \\

The above allows us to express the martingale $M(t)(\phi)$ in \eqref{MP} as an integral of $\phi$ with respect to the martingale measure $M(ds, dx)$:
\begin{equation*}
M(t)(\phi) = \int_0^t \int_E \phi(x) M(ds, dx).
\end{equation*}

From \eqref{eq.nu} we get that the integral $f \bullet M$, with $M$ being the martingale measure derived from \eqref{MP}, is defined for all functions $f$ that are predictable and satisfy
\begin{equation}
\mathbb{E}\left[ c \int_0^T \int_E | f (s,x) |^2 X(s)(dx) ds \right] < \infty  ,
\label{eq.finit}
\end{equation}
where the expectation is with respect to the measure $\mathbb{P}_m$. \\

As mentioned above, following standard arguments known from the theory of stochastic integrals with respect to a martingale (see e.g. Chapter 3 \cite{ChWi14}), we can extend the class of integrands for which the stochastic integral with respect to a martingale measure is defined to include bounded optional functions, as long as the measure $\nu(\omega,dt,dx)$ induced by the covariance is in $dt$ absolutely continuous w.r.t. to Lebesque measure $dt$ for almost all $\omega$.

\begin{Def}[Optional Functions]
The optional-$\sigma$-algebra $\mathcal{O}$ is the $\sigma$-algebra generated by functions of form
\begin{equation*}
f(\omega ,s,x) = X(\omega) 1_{[a , b)}(s) 1_A(x)
\end{equation*}
with $0 \leq a < b \leq T$, bounded and $\mathcal{F}_a$-measurable $X$ as well as  $A \in \mathcal{E}$. A function is called optional if it is $\mathcal{O}$-measurable. 
\end{Def}

Define a measure $\mu_M$ on $\mathcal{F} \times \mathcal{B}([0,T]) \times \mathcal{E}$ by \begin{equation*}
\mu_M (A_1 \times A_2 \times A_3) = \mathbb{E} \left[ c \int_0^T \int_E 1_{A_1 \times A_2 \times A_3} (\omega, s,x) X(\omega, s)(dx) ds\right] ,
\end{equation*}
where the expectation is with respect to $\mathbb{P}_m$. This is the extension of the so-called Doléans measure of $M$ to $\mathcal{F} \times \mathcal{B}([0,T]) \times \mathcal{E}$. Further, let $\mathcal{L}^2_\mathcal{P} = L^2(\Omega \times [0,T] \times E, \mathcal{P}, \mu_M)$ be the space of $\mathcal{P}$-measurable, i.e. predictable, functions satisfying $
\int f^2 d\mu_M < \infty.$ Then $\mathcal{L}^2_{\mathcal{P}}$ coincides with $\mathcal{P}_M$ for the $B(A,c)$-superprocess. Finally, denote the class of $\mu_M$-null sets in $\mathcal{F} \times \mathcal{B}([0,T]) \times \mathcal{E}$ by $\mathcal{N}$ and set $\tilde{\mathcal{P}} = \mathcal{P} \wedge \mathcal{N}$. In the remainder of this subsection we show that optional processes are $\tilde{\mathcal{P}}$-measurable and that $\mathcal{L}^2_{\mathcal{P}}= \mathcal{L}^2_{\tilde{\mathcal{P}}}$. To do so, we need the following proposition.

\begin{Prop} \label{AuxProp}
\begin{enumerate}[label = (\roman*)]
\item A subset $A$ of $\Omega \times [0,T] \times E$ belongs to $\tilde{\mathcal{P}}$ if and only if there exists a $B \in \mathcal{P}$ such that
\begin{equation*}
A \Delta B = (A \setminus B) \cup (B \setminus A) \in \mathcal{N}.
\end{equation*}
\item If $f : \Omega \times [0,T] \times E \rightarrow \mathbb{R}$ is $\mathcal{F} \times \mathcal{B}([0,T]) \times \mathcal{E}$-measurable, it is $\tilde{\mathcal{P}}$-measurable if and only if there exists a predictable function $g$ such that 
\begin{equation*}
\{ f \neq g\} \in \mathcal{N}.
\end{equation*}
\end{enumerate} 
\end{Prop}

\begin{proof}
\begin{enumerate}[label = (\roman*)]
\item Set 
\begin{equation*}
\mathcal{A} = \{ A \ | \ \exists B \in \mathcal{P} \ \text{s.t.} \ A \Delta B \in \mathcal{N}  \}
\end{equation*}
As $C = A \Delta B$ if and only if $A = C \Delta B$ and since $C \in \mathcal{N} \subset \tilde{\mathcal{P}}$, $B \in \mathcal{P} \subset \tilde{\mathcal{P}}$, we have $A = C \Delta B \in \tilde{\mathcal{P}}$. Consequently $\mathcal{A} \subset \tilde{\mathcal{P}}$. \\

To prove that $\tilde{\mathcal{P}} \subset \mathcal{A}$ also holds, we only have to check if $\mathcal{A}$ is a $\sigma$-algebra since it contains $\mathcal{N}$ and $\mathcal{P}$. However, as 
\begin{itemize}
\item $\emptyset \in \mathcal{P}$, $\emptyset \in \mathcal{N}$, $\emptyset \Delta \emptyset = \emptyset$, we get that $\emptyset \in \mathcal{A}$,
\item $A^c \Delta B^c = A \Delta B$, we get $\mathcal{A}^c \in \mathcal{A}$ for all $A \in \mathcal{A}$,
\item $\left( \bigcup_{n} A_n \right) \Delta \left( \bigcup_n B_n \right) \subset \bigcup_n ( A_n \Delta B_n)$, we get that $\bigcup_n A_n \in \mathcal{A}$ for all $(A_n)_n \subset \mathcal{A}$,
\end{itemize}
this is the case. Thus $\tilde{\mathcal{P}} \subset \mathcal{A}$ and therefore $\tilde{\mathcal{P}} = \mathcal{A}$.
\item Let $g$ be predictable and assume $\{ f \neq g \} \in \mathcal{N}$. Then we have
\begin{equation*}
\{f \in S\} \Delta \{g \in S\} \subset \{f \neq g \} \quad \text{for all} \ S \in \mathcal{B}(\mathbb{R}).
\end{equation*}
By part (i), as $\{g \in S\} \in \mathcal{P}$, we get $\{f \in S\} \in \tilde{\mathcal{P}}$ and consequently $f$ is $\tilde{\mathcal{P}}$-measurable. \\

Now assume $f$ is $\tilde{\mathcal{P}}$-measurable and $f = \sum_{j=1}^\infty c_j 1_{A_j}$ for disjoint $A_j \in \tilde{\mathcal{P}}$ and $c_j \in \mathbb{R}$. For each $A_j$, there exists a $B_j \in \mathcal{P}$ such that $A_j \Delta B_j \in \mathcal{N}$ by part (i). As the $A_j$'s are disjoint, we have $B_i \cap B_j \in \mathcal{N}$ if $i \neq j$. To obtain disjoint sets, set 
\begin{equation*}
B'_1 = B_1 \quad \text{and} \quad B'_j = \left( \bigcap B_i^c \right) \cap B_j \ \text{for $j \geq 2$}.
\end{equation*}
Then $B_j \Delta B'_j \in \mathcal{N}$, $A_j \Delta B'_j \in \mathcal{N}$ and $B'_i \cap B'_j = \emptyset$ if $i \neq j$. Next, set $g = \sum_{j=1}^\infty c_j 1_{B'_j}$. Then $g$ is $\mathcal{P}$-measurable and 
\begin{equation*}
\{ f \neq g \} \subset \bigcup_{j=1}^\infty (A_j \Delta B'_j) \in \mathcal{N}.
\end{equation*}
For general $\tilde{\mathcal{P}}$-measurable $f$, there exists a sequence of $\tilde{\mathcal{P}}$-measurable functions $f^n$ of the above form such that $f^n \rightarrow f$ $\mu_M$-almost surely. Pick $\mathcal{P}$-measurable $g^n$ such that $\{f^n \neq g^n\} \in \mathcal{N}$ and set $g = \lim \inf_{n \rightarrow \infty} g^n$.  Then $g$ is also $\mathcal{P}$-measurable and
\begin{equation*}
\bigcup_{n = 1}^\infty \{ f^n \neq g^n \} \cup \{ \lim_{n \rightarrow \infty} f^n \neq f\} \in \mathcal{N} .
\end{equation*}
Set 
\begin{equation*}
\Sigma = (\Omega \times [0,T] \times E) \setminus \left(\bigcup_{n = 1}^\infty \{ f^n \neq g^n \} \cup \{ \lim_{n \rightarrow \infty} f^n \neq f\}\right) .
\end{equation*}
Obviously, the set $\Sigma$ has full mass, i.e. $(\Omega \times [0,T] \times E) \setminus \Sigma \in \mathcal{N}$, and on $\Sigma$ we have $f^n(\omega, s, x) = g^n(\omega, s, x)$ for all $n \in \mathbb{N}$. Therefore 
\begin{equation}
\liminf_{n \rightarrow \infty} f^n(\omega, s, x) = \liminf_{n \rightarrow \infty} g^n(\omega, s, x)
\end{equation} 
on $\Sigma$ and
\begin{align*}
\emptyset &= \{( \omega, s, x) \in \Sigma :  f(\omega , s, x) \neq \lim_{n \rightarrow \infty} f^n(\omega, s, x) \} \\
&= \{( \omega, s, x) \in \Sigma :  f(\omega , s, x) \neq \liminf_{n \rightarrow \infty} f^n(\omega, s, x) \}  \\
&= \{( \omega, s, x) \in \Sigma  :  f(\omega , s, x) \neq \liminf_{n \rightarrow \infty} g^n(\omega, s, x) \} .
\end{align*}
Thus $\{( \omega, s, x) \in \Sigma  :  f(\omega , s, x) \neq \liminf_{n \rightarrow \infty} g^n(\omega, s, x) \} \in \mathcal{N}$ and we obtain 
\begin{align*}
& \{ (\omega, s, x) \in \Omega \times [0,T] \times E : f(\omega, s, x) \neq \liminf_{n \rightarrow \infty} g^n(\omega, s,x)\} \\
\subset \   & \{( \omega, s, x) \in \Sigma  :  f(\omega , s, x) \neq \liminf_{n \rightarrow \infty} g^n(\omega, s, x) \} \cup (\Omega \times [0,T] \times E) \setminus \Sigma .
\end{align*}
As both sets on the right hand side are null sets, we get 
\begin{equation*}
\{ f \neq g\} \in \mathcal{N},
\end{equation*}
which completes the proof.
\end{enumerate}
\end{proof}

The above proposition allows us to prove the following.

\begin{Prop}
Any bounded optional function is $\tilde{\mathcal{P}}$-measurable.
\end{Prop}
\begin{proof}
 Let $f(\omega, s ,x) = X(\omega) 1_{[a,b)}(s) 1_A(x)$ and $g(\omega, s ,x) = X(\omega) 1_{(a,b]}(s) 1_A(x)$ with $0 \leq a < b \leq T$, $X$ $\mathcal{F}_a$-measurable and $A \in \mathcal{E}$. As $g$ is predictable, we get from Proposition \ref{AuxProp} that $f$ is $\tilde{\mathcal{P}}$-measurable if $\{ f \neq g\} \in \mathcal{N}$. This is the case as
\begin{align*}
\mu_M (\{ f \neq g \}) &= \mathbb{E} \left[ c \int_0^T \int_E 1_{\{ f \neq g \}} X(s)(dx) ds \right] \\
&= \mathbb{E} \left[ c \int_0^T \int_E 1_{\{ X 1_{[a,b)} (s) 1_{A}(x) \neq X 1_{(a,b]}(s) 1_A(x)\}} X(s)(dx) ds \right] \\
&= \mathbb{E} \left[ c \int_0^T \int_A 1_{\{ \Omega \times [a] \times A \}} X(s)(dx)ds \right] \\
&= \mathbb{E} \left[ c \int_0^T 1_{[a]} X(s)(A) ds \right] \\
&= 0.
\end{align*}
The last equality holds as $X(s)(A)$ is almost surely finite. Consequently $f$ is $\tilde{\mathcal{P}}$-measurable and, as functions of this form generate $\mathcal{O}$, we get that optional functions are $\tilde{\mathcal{P}}$-measurable.  
\end{proof}

 Since $\mathcal{P} \subset \tilde{\mathcal{P}}$, we have $\mathcal{L}_{\mathcal{P}}^2 \subset \mathcal{L}_{\tilde{\mathcal{P}}}^2$. From Proposition \ref{AuxProp} we get the existence of a function $g \in \mathcal{L}^2_{\mathcal{P}}$ for every $f \in \mathcal{L}^2_{\tilde{\mathcal{P}}}$ such that $f = g$ $\mu_M$-a.e.. Therefore the Hilbert spaces $\mathcal{L}_{\mathcal{P}}^2$ and $\mathcal{L}_{\tilde{\mathcal{P}}}^2$ are equal. Consequently, we can extend the stochastic integral with respect to the martingale measure associated with the $B(A,c)$-superprocess to square-integrable, $\tilde{\mathcal{P}}$-measurable integrands, which include square-integrable optional functions. The definition of the integral as the limit of integrals of simple functions is still valid and for simplicity, we also denote the extended class of integrands by $\mathcal{P}_M$. \\

Extending the class of integrands to include square-integrable optional functions is of interest because one can show that $\sigma(r.c.l.l) = \mathcal{O}$, where $r.c.l.l.$ are the adapted right continuous functions with left limits. Being able to define the stochastic integral with respect to the martingale measure for adapted right continuous processes with left limits is crucial to prove Theorem \ref{Thm.fct.Ito}. As we are only interested in the fact that $\sigma(r.c.l.l.) \subset \mathcal{O}$ and as the other inclusion follows immediately from the definition of $\mathcal{O}$, we only show the following.

\begin{Prop} It holds: 
\begin{enumerate}[label = (\roman*)]
\item  If $f$ is $\mathcal{F}_a \times \mathcal{E}$-measurable, then $f 1_{[a,b)}$ is optional for all $0 \leq a < b \leq T$.
\item $\sigma(r.c.l.l.) \subset \mathcal{O}$.
\end{enumerate}
\end{Prop}

\begin{proof}
\begin{enumerate}[label = (\roman*)]
\item As $f$ is $\mathcal{F}_a \times \mathcal{E}$-measurable, we can approximate it pointwise by sums of indicator functions $1_{A \times F}$, $F \in \mathcal{F}_a$, $A \in \mathcal{E}$. Therefore, it is enough to consider $f = 1_{F \times A}$. As we can write
\begin{equation*}
f 1_{[a,b)}(\omega, s, x) = 1_F(\omega) (\omega) 1_{[a,b)}(s) 1_A(x),
\end{equation*}
the function $f 1_{[a,b)}$ is optional as $X(\omega) = 1_F(\omega)$ is $\mathcal{F}_a$-measurable.
\item Let $f$ be an adapted r.c.l.l. function. Consider a partition $\{t_i^n : 0 \leq i \leq n , 1 \leq n \leq \infty\}$ with $t_0^n = 0$, $t_n^n = T$, $t_i^n < t_{i+1}^n$ for all $n$ and $\max_i | t_{i+1}^n - t_i^n | \searrow 0$ as $n \rightarrow \infty$. Further, define an approximation of the function $f$ by 
\begin{equation*}
App^n(f)(\omega, s, x) = \sum_{i=0}^{n-1} f(\omega, t_i^n , x) 1_{[t_i^n, t_{i+1}^n)}(s) .
\end{equation*}
This approximation converges pointwise to $f$ and as every summand of the above is optional by part (i) so is the approximation and thus the limit $f$ is also optional. 
\end{enumerate}
\end{proof}

\subsection{The It\={o}-Formula}

Before getting to the more general case, let us first prove our main result for a simple class of functions before presenting the more general result in Theorem \ref{Thm.Ito}. 

\begin{Thm} \label{Thmfb}
Let the function $F:[0, T] \times M_F(E) \rightarrow \mathbb{R}$ be finitely based, i.e. there exist a bounded $f \in C^2([0,T] \times \mathbb{R}^{n}, \mathbb{R})$ with bounded derivatives and $\phi_1 , \ldots , \phi_n \in D(A)$ such that 
\begin{equation*}
F(t, \mu) = f(t, \langle\mu , \phi_1 \rangle , \ldots , \langle\mu , \phi_n \rangle ) .
\end{equation*}
Further let $A$ be a good generator and $X$ solves \eqref{MP}. Then the following It\={o}-type formula holds:
\begin{align*}
\begin{split}
F(t, X(t)) - F(0 , X(0)) = \quad &\hphantom{\frac{1}{2}}   \int_0^t D_s F(s, X(s)) ds \\
+ \ & \hphantom{\frac{1}{2}} \int_0^t \int_E A^{(x)} D_x F(s, X(s)) X(s)(dx) ds \\
+ \ & \frac{1}{2} \int_0^t \int_E c D_{xx} F(s, X(s)) X(s)(dx) ds \\
+ \ &\hphantom{\frac{1}{2}}  \int_0^t \int_E  D_x F(s, X(s)) M(dx, ds).
\end{split} 
\end{align*}
\end{Thm}

\begin{proof}
As $F$ is finitely based and as the $\langle X(t) , \phi_i \rangle$ are semimartingales, the traditional It\={o}-formula for semimartingales yields, with $\partial_i f$ denoting partial derivatives,
\begin{align}
\begin{split}
&f(t, \langle X(t) , \phi_1 \rangle, \ldots , \langle X(t) , \phi_n \rangle) \\
= \quad & f(0, \langle X(0) , \phi_1 \rangle, \ldots , \langle X(0) , \phi_n \rangle) \\
+ & \int_0^t \partial_s f(s, \langle X(s) , \phi_1 \rangle, \ldots , \langle X(s) , \phi_n \rangle) ds \\
+ & \int_0^t \sum_{i = 1}^n \partial_i f(s, \langle X(s) , \phi_1 \rangle, \ldots , \langle X(s) , \phi_n \rangle) d\langle X(t) , \phi_i \rangle \\
+ & \int_0^t \sum_{i,j = 1}^n \partial_{ij}f(s, \langle X(s) , \phi_1 \rangle, \ldots , \langle X(s) , \phi_n \rangle) d[\langle X , \phi_i \rangle, \langle X , \phi_j \rangle]_s .
\end{split}
\label{itosem}
\end{align}
As $X$ is a solution to \eqref{MP}, 
\begin{equation*}
\langle X(s), \phi_i \rangle = M(s)(\phi_i) + \langle X(0) , \phi_i \rangle + \int_0^s \langle X(r) , A \phi_i \rangle dr 
\end{equation*}
and, as $\langle X(0) , \phi_i \rangle$ is constant, 
\begin{align*}
d\langle X(s), \phi_i \rangle =  d(M(s)(\phi_i))  +  \langle X(r) , A \phi_i \rangle dr.
\end{align*}
By the properties of the martingale measure it holds 
\begin{equation*}
M(t)(\phi) = \int_0^t \int_E \phi(x) M(ds, dx)
\end{equation*}
and 
\begin{equation*}
[M(\phi_1), M(\phi_2)]_t = \int_0^t \langle X(s), c \phi_1 \phi_2 \rangle ds
\end{equation*}
and thus 
\begin{equation*}
d\langle X(s) , \phi_i \rangle  =  \int_E \phi_i(x) M(ds, dx) +  \langle X(s) , A \phi_i \rangle ds
\end{equation*}
and
\begin{equation*}
d [\langle X , \phi_i \rangle, \langle X , \phi_j \rangle]_s = \langle X(s), c \phi_i \phi_j \rangle ds.
\end{equation*}
Plugging these terms into \eqref{itosem} yields
\begin{align*}
&f(t, \langle X(t) , \phi_1 \rangle, \ldots , \langle X(t) , \phi_n \rangle) \\
= \quad & f(0, \langle X(0) , \phi_1 \rangle, \ldots , \langle X(0) , \phi_n \rangle) \\
+ & \int_0^t \partial_s f(s, \langle X(s) , \phi_1 \rangle, \ldots , \langle X(s) , \phi_n \rangle) ds \\
+ & \int_0^t \sum_{i = 1}^n \partial_i f(s, \langle X(s) , \phi_1 \rangle, \ldots , \langle X(s) , \phi_n \rangle) \int_E \phi_i (x) M(ds, dx) \\
+ & \int_0^t \sum_{i = 1}^n \partial_i f(s, \langle X(s) , \phi_1 \rangle, \ldots , \langle X(s) , \phi_n \rangle) \langle X(s) , A \phi_i \rangle ds \\
+ & \int_0^t \sum_{i,j = 1}^n \partial_{ij}f(s, \langle X(s) , \phi_1 \rangle, \ldots , \langle X(s) , \phi_n \rangle) \langle X(s), c \phi_1 \phi_2 \rangle ds .
\end{align*}
The result now follows by computing the directional derivatives of finitely based functions as in the statement of the theorem and identification with the expression above. For the linear function $F_\phi (\mu) = \langle \phi, \mu \rangle$, $D_xF_\phi(\mu) = \phi(x)$ holds, which yields 
\begin{equation*}
\langle \mu , A \phi \rangle  = \langle \mu , AD_\cdot F_\phi \rangle .
\end{equation*}
Using this and the directional derivatives, we obtain from the chain rule of ordinary differentiation, 
\begin{equation*}
D_xF(t, \mu) = \sum_{i=1}^n \partial_i f(s, y_1 , \ldots, y_n)  | _{y_1 = \langle \mu , \phi_1 \rangle , \ldots , y_n = \langle \mu , \phi_n \rangle }\phi_i(x) ,
\end{equation*}
that, for example, 
\begin{align*}
&\int_0^t \int_E A (D F(s, X(s)))(x) X(s)(dx) ds \\
= \ & \int_0^t \int_E A \left(\sum_{i= 1}^n \partial_i f(s, \langle X(s) , \phi_1 \rangle, \ldots , \langle X(s) , \phi_n \rangle)(D F_{\phi_i}(X(s))\right)(x) X(s) (dx) ds \\
= \ & \int_0^t \sum_{i=1}^n \partial_i f(s, \langle X(s) , \phi_1 \rangle, \ldots , \langle X(s) , \phi_n \rangle) \int_E A \phi_i(x) X(s) (dx) ds \\
= \ & \int_0^t \sum_{i=1}^n  \partial_i f(s, \langle X(s) , \phi_1 \rangle, \ldots , \langle X(s) , \phi_n \rangle) \langle X(s), A \phi_i \rangle ds
\end{align*}
and
\begin{align*}
& \int_0^t \int_E D_x F(s, X(s)) M(ds, dx) \\
= \ & \int_0^t \int_E \sum_{i=1}^n \partial_i f(s, \langle X(s) , \phi_1 \rangle, \ldots , \langle X(s) , \phi_n \rangle) (D_x \langle X(s), \phi_i \rangle) M(ds, dx) \\
= \ & \int_0^t \sum_{i=1}^n\partial_i f(s, \langle X(s) , \phi_1 \rangle, \ldots , \langle X(s) , \phi_n \rangle) \int_E \phi_i(x) M(ds, dx) .
\end{align*}
The integral is well defined as $\partial_i f$ and $\phi$ are continuous and bounded. The equivalence of the remaining terms follows equivalently.
\end{proof}

Now, for functions $F : [0,T] \times M_F(E) \rightarrow \mathbb{R}$ define the following set of conditions. 
\begin{Cond}\label{cond1}
\begin{enumerate}[label = (\roman*)]
\item $F(s, \mu)$, $D_x F(s, \mu)$, $D_{xy}F(s, \mu)$, $D_{xyz}F(s, \mu)$, $D_s F(s, \mu)$, $D_{sx}F(s, \mu)$, $D_{sxy}F(s, \mu)$ and $D_{sxyz}F(s, \mu)$ exists and are continuous in $x$, $y$, $z$, $\mu$ and $s$, 
\item for fixed $s$, $y$, $z$ and $\mu$, the maps $x \mapsto D_x F(s, \mu)$, $x \mapsto D_{xy}F(s , \mu)$ and $x \mapsto D_{xyz} F(s , \mu)$ are in the domain of the generator $A$, 
\item $A^{(x)} D_x F(s,\mu)$, $A^{(x)} D_{xy} F(s,\mu)$ and $A^{(x)}D_{xyz}F(s ,\mu)$ are continuous in $s$, $x$, $y$, $z$ and $\mu$.
\end{enumerate}
\end{Cond}

For a processes $X$ satisfying \eqref{MP} and such functions $F$, we have have the following It\={o}-formula.

\begin{Thm} \label{Thm.Ito}
If $X$ is a solution to the martingale problem \eqref{MP}, $A$ is a good generator and $F$ satisfies Conditions \ref{cond1}, then
\begin{align}
\begin{split}
F(t, X(t)) - F(0 , X(0)) = \quad &\hphantom{\frac{1}{2}}   \int_0^t D_s F(s, X(s)) ds \\
+ \ & \hphantom{\frac{1}{2}} \int_0^t \int_E A^{(x)} D_x F(s, X(s)) X(s)(dx) ds \\
+ \ & \frac{1}{2} \int_0^t \int_E c D_{xx} F(s, X(s)) X(s)(dx) ds \\
+ \ &\hphantom{\frac{1}{2}}  \int_0^t \int_E  D_x F(s, X(s)) M(dx, ds).
\end{split} 
\label{SP-Ito}
\end{align}
\end{Thm}

\begin{proof}
From Theorem 3 in \cite{JT03} we get for $t \in [0,T]$
\begin{align*}
F(t, X(t)) - F(0 , X(0)) = \quad & \hphantom{\frac{1}{2}} \int_0^t D_s F(s, X(s)) ds \\
+  \ & \hphantom{\frac{1}{2}}  \int_0^t \int_E A^{(x)} D_x F(s, X(s)) X(s)(dx) ds \\
+ \ & \frac{1}{2} \int_0^t \int_E c D_{xx} F(s, X(s)) X(s)(dx) ds \\
+  \ & \hphantom{\frac{1}{2}}  \tilde{M} ,
\end{align*}
where $\tilde{M}$ is a local martingale. \\ 

It remains to be proved that $\tilde{M}$ has the stated form as an integral with respect to the martingale measure. From the proof in \cite{JT03}, we get the existence of finitely based functions $F^n$ such that the following holds with $M_K(E) = \{ \mu \in M_F(E) \ : \ \langle \mu , 1 \rangle \leq K \}$.
\begin{align}
\begin{split}
\sup_{t \in [0,T], \ \mu \in M_K(E)} &| F(t , \mu) - F^n(t, \mu)| \rightarrow 0, \\ 
\sup_{t \in [0,T], \ x,y \in E , \ \mu \in M_K(E)} &|  D_{xy} F(t , \mu) - D_{xy} F^n(t, \mu)| \rightarrow 0, \\
\sup_{t \in [0,T], \ \mu \in M_K(E)} &| D_t F(t, \mu) + \int_E A^{(z)} D_z F(t,\mu) \mu(dz) \\
 & \ - D_t F^n(t, \mu) - \int_E A^{(z)} D_z F^n(t,\mu) \mu(dz) | \rightarrow 0 .
\end{split}
\label{ConvFromJT}
\end{align}
To see that it suffices to consider the case $\mu \in M_K(E)$, consider the stopping time $\tau  = \inf \{t \geq 0 : \langle X(t) , 1 \rangle \geq K \}$. Then, the result holds when we replace $t$ by $t \wedge \tau$ in the result and as $\tau \rightarrow \infty$ if $K \rightarrow \infty$, we get that $t \wedge \tau \rightarrow t $ for all $t \in [0,T]$ as $K \rightarrow \infty$. \\ 

As $F^n$ is finitely based, we have from the It\={o}-formula for finitely based functions (Theorem \ref{Thmfb})
\begin{align}
\begin{split}
F^n(t, X(t)) - F^n(0 , X(0)) = \quad & \hphantom{\frac{1}{2}}  \int_0^t D_s F^n(s, X(s)) ds \\
+   \ & \hphantom{\frac{1}{2}} \int_0^t \int_E A^{(x)} D_x F^n(s, X(s)) X(s)(dx) ds \\
+ \ & \frac{1}{2} \int_0^t \int_E c D_{xx} F^n(s, X(s)) X(s)(dx) ds \\
+  \ & \hphantom{\frac{1}{2}} \int_0^t \int_E  D_x F^n(s, X(s)) M(dx, ds).
\end{split}
\label{eq1proof2}
\end{align}

Combining \eqref{ConvFromJT} and \eqref{eq1proof2}, we get
\begin{align*}
F(t, X(t)) &= \lim_{n \rightarrow \infty} F^n(t, X(t)) \\
&= \lim_{n \rightarrow \infty} \left( \vphantom{\int_0^1} \right. F^n(0, X(0)) \\
& \hphantom{\lim_{n \rightarrow \infty} (} + \int_0^t  D_s F^n(s, X(s)) + \int_E A^{(x)} D_x F^n(s, X(s)) X(s)(dx) ds \\
& \hphantom{\lim_{n \rightarrow \infty} (} + \frac{1}{2} \int_0^t \int_E c D_{xx} F^n(s, X(s)) X(s)(dx)ds \\
& \hphantom{\lim_{n \rightarrow \infty} (} + \int_0^t \int_E D_x F^n(s, X(s)) X(s)(dx)ds \left. \vphantom{\int_0^1}  \right) \\
&=F(0, X(0)) + \int_0^t  D_s F(s, X(s)) ds  \\
& \hphantom{\lim_{n \rightarrow \infty} (} + \int_0^t \int_E A^{(x)} D_x F(s, X(s)) X(s)(dx) ds \\
& \hphantom{\lim_{n \rightarrow \infty} (} + \frac{1}{2} \int_0^t \int_E c D_{xx} F(s, X(s)) X(s)(dx)ds \\
& \hphantom{\lim_{n \rightarrow \infty} (} + \lim_{n \rightarrow \infty} \int_0^t \int_E D_x F^n(s, X(s)) X(s)(dx)ds .
\end{align*}
As a last step, we show that  $D_x F^n$ converges to $D_x F$ uniformly. Lemma 4 in \cite{JT03} states the following. Let es $G: M_F(E) \rightarrow \mathbb{R}$ be continuous with continuous derivative $D_xG$. Then 
\begin{equation*}
G(\mu) = G(0) + \int_0^1 \int_E D_x G(\theta \mu) \mu(dx) d \theta.
\end{equation*} 
This allows us to use the convergence of the second order derivative of $F^n$ to obtain
\begin{align*}
&\sup_{\substack{t \in [0,T], \,  x\in E \\ \mu \in M_K(E)}} |  D_x F(t, \mu) - D_x F^n(t, \mu) | \\
\leq &\sup_{\substack{t \in [0,T], \,  x\in E \\ \mu \in M_K(E)}} | \int_0^1 \int_E D_{xy} F(t, \theta \mu) \mu(dy) d\theta - \int_0^t \int_E D_{xy} F^n(t ,\theta \mu) \mu(dy) d \theta | \\
& \quad + \sup_{\substack{t \in [0,T], \,  x\in E \\ \mu \in M_K(E)}} |  D_x F(t,0) - D_x F^n(t, 0)|\\
= &\sup_{\substack{t \in [0,T], \,  x\in E \\ \mu \in M_K(E)}} | \int_0^1 \int_E D_{xy} F(t, \theta \mu) -  D_{xy} F^n(t ,\theta \mu) \mu(dy) d \theta | \\
& \quad + \sup_{\substack{t \in [0,T], \,  x\in E \\ \mu \in M_K(E)}} |  D_x F(t,0) - D_x F^n(t, 0)|\\
\leq  &\quad \int_0^1 \int_E \sup_{\substack{t \in [0,T], \,  x, \, y \in E \\ \mu \in M_K(E)}} |  D_{xy} F(t, \theta \mu) -  D_{xy} F^n(t ,\theta \mu) | \mu(dy) d \theta  \\
& \quad + \sup_{\substack{t \in [0,T], \,  x\in E \\ \mu \in M_K(E)}} |  D_x F(t,0) - D_x F^n(t, 0)| \\
& \quad \xrightarrow{n \rightarrow \infty} 0 .
\end{align*}

The convergence in the sup-norm yields the convergence in $\| \cdot \|_M$ and thus the convergence of the stochastic integrals with respect to the martingale measure,
\begin{equation*}
\lim_{n \rightarrow \infty} \int_0^t \int_E D_x F^n(s, X(s))M(dx, ds) = \int_0^t \int_E D_x F(s, X(s))M(dx, ds).
\end{equation*} 
\end{proof}

\section{A Functional It\={o}-Formula for the Dawson-Watanabe Superprocess} \label{SecHP}

In Theorem \ref{Thm.Ito}, we consider functions $F$ of $X(t)$, i.e. of the value of $X$ at a specific time $t$. In what follows, we look at functionals $F$ of the path of the process $X$. The following is motivated by the results  in \cite{Co16}. \\

To formalize the notion of (stopped) paths, consider an arbitrary path $\omega \in D([0,T], M_F(E))$, the space of right continuous functions with left limits. We equip $D([0,T], M_F(E))$ with the metric $\tilde{d}$ given by 
\begin{equation*}
\tilde{d}(\omega, \omega') = \sup_{u \in [0,T]} d_P( \omega (u), \omega'(u))
\end{equation*}
for $\omega$, $\omega' \in D([0,T], M_F(E))$, where $d_P$ is the Prokhorov metric on $M_F(E)$. For such a path, define $\omega_t$, the path stopped at time $t$, by $\omega_t(u) = \omega ( t \wedge u)$ for $u \in [0,T]$. Further, define $\omega_{t-}$ by $\omega_{t-}(u) = \omega(u)$ for $u \in [0,t)$ and $\omega_{t-}(u) = \omega(t-)$ for $u \in [t,T]$. Using this, we define an equivalence relation on the space $[0,T] \times D([0,T], M_F(E))$ by 
\begin{equation*}
(t ,  \omega) \sim (t' , \omega') \quad \Leftrightarrow \quad t = t' \ \text{and} \ \omega_t = \omega'_{t'} .
\end{equation*} 
This relation gives rise to the quotient space 
\begin{equation*}
\Lambda_T = \{ (t , \omega_t) \ | \ (t, \omega) \in [0,T] \times D([0,T], M_F(E)) \} \equiv [0,T] \times D([0,T], M_F(E)) / \sim .
\end{equation*}
Next, define a metric $d_\infty$ on $\Lambda_T$ by
\begin{equation*}
d_\infty((t, \omega), (t', \omega')) = \sup_{u \in [0,T]} d_P (\omega(u \wedge t), \omega'(u \wedge t')) + | t - t'|.
\end{equation*}

A functional $F : (\Lambda_T,d_\infty) \rightarrow \mathbb{R}$ is continuous with respect to $d_\infty$ if for all $(t, \omega) \in \Lambda_T$ and every $\varepsilon > 0$ there exists an $\eta > 0$ such that for all $(t', \omega') \in \Lambda_T$ with $d_\infty((t,\omega), (t', \omega')) < \eta$ we have
\begin{equation*}
|F(t, \omega) - F(t' , \omega')| < \varepsilon .
\end{equation*}

Further, a functional $F$ on $[0,T] \times D([0,T], M_F(E))$ is called \emph{non-anticipative} if it is a measurable map on the space of stopped paths, i.e. $F: (\Lambda_T,\mathcal{B}(\Lambda_T) )  \rightarrow (\mathbb{R},\mathcal{B}(\mathbb{R}))$. In other words, $F$ is non-anticipative if $F (t, \omega) = F(t, \omega_t)$ holds for all $\omega \in D([0,T], M_F(E))$. \\

For continuous non-anticipative functionals, we can define two types of derivatives.

\begin{Def}
A continuous non-anticipative functional $F : \Lambda_T \rightarrow \mathbb{R}$ is
\begin{enumerate}[label = (\roman*)]
\item  \emph{horizontally differentiable} at $(t, \omega) \in \Lambda_T$ if the limit
\begin{equation*}
\mathcal{D}^* F (t, \omega) = \lim_{ \varepsilon \rightarrow 0} \frac{F(t + \varepsilon, \omega_t) - F(t, \omega_t)}{\varepsilon}
\end{equation*}
exists. If this is the case for all $(t, \omega) \in \Lambda_T$, we call $\mathcal{D}^* F$ the \emph{horizontal derivative of $F$}.

\item \emph{vertically differentiable} at $(t, \omega) \in \Lambda_T$ in direction $x \in E$ if the limit
\begin{equation*}
\mathcal{D}_x F(t, \omega) = \lim_{\varepsilon \rightarrow 0} \frac{F(t, \omega_t + \varepsilon \delta_x 1_{[t,T]}) - F(t, \omega_t)}{\varepsilon}
\end{equation*}
exists. If this is the case for all $(t, \omega) \in \Lambda_T$, we call $\mathcal{D}_xF$ the \emph{vertical derivative of $F$ in direction $x$}. Higher order vertical derivatives are defined iteratively.
\end{enumerate}
\end{Def}

The definition of the derivatives allows us to define the following set of conditions for a functional $F: \Lambda_T \rightarrow \mathbb{R}$.

\begin{Cond}\label{cond2}
\begin{enumerate}[label = (\roman*)]
\item $F$ is bounded and continuous,
\item the horizontal derivative $\mathcal{D}^*F(t, \omega)$ is continuous and bounded in $t$ and $\omega$,
\item the vertical derivatives $\mathcal{D}_{x_1}F(t, \omega)$, $\mathcal{D}_{x_1 x_2}F(t, \omega)$, $\mathcal{D}_{x_1 x_2 x_3}F(t, \omega)$ and the mixed derivatives\footnote{$\mathcal{D}^*_{x}F(t, \omega) = \mathcal{D}^*(\mathcal{D}_x F(t, \omega))$} $\mathcal{D}^*_{x_1}F(t, \omega)$, $\mathcal{D}^*_{x_1 x_2}F(t, \omega)$, $\mathcal{D}^*_{x_1 x_2 x_3}F(t, \omega)$ are bounded and continuous in $t$, $\omega$, $x_1$, $x_2$ and $x_3$,
\item for fixed $t$, $x_1$, $x_2$, $\omega$ the maps $x \mapsto \mathcal{D}_x F(t, \omega)$ , $x \mapsto \mathcal{D}_{x x_1} F(t, \omega)$ and $x \mapsto \mathcal{D}_{x x_1 x_2} F(t, \omega)$ are in the domain of $A$, 
\item $A^{(x)} \mathcal{D}_{x_1} F(t, \omega)$, $A^{(x)} \mathcal{D}_{x_1 x_2} F(t, \omega)$ and $A^{(x)} \mathcal{D}_{x_1 x_2 x_3} F(t, \omega)$ are continuous in $t$, $x_1$, $x_2$, $x_3$ and $\omega$.
\end{enumerate}
\end{Cond}

For functionals $F$ satisfying these conditions, we can now formulate a functional It\={o}-formula for the $B(A,c)$-superprocess. Note that, as the $B(A,c)$-superprocess has continuous paths, its stopped paths $X_t$ are in $C([0,T], M_F(E))$, the space of continuous mappings from $[0,T]$ to $M_F(E)$.

\begin{Thm} \label{Thm.fct.Ito}
Assume $F$ satisfies Conditions \ref{cond2}, $A$ is a good generator and $X$ is a $B(A,c)$-superprocess. Then for $t \in [0,T]$
\begin{align*}
F(t, X_t) - F(0 , X(0)) \ = \quad &\hphantom{\frac{1}{2}}   \int_0^t \mathcal{D}^* F( s, X_s) ds \\
+ \ & \hphantom{\frac{1}{2}} \int_0^t \int_E A^{(x)} \mathcal{D}_x F(s, X_s) X(s)(dx) ds \\
+ \ & \frac{1}{2} \int_0^t \int_E c \mathcal{D}_{xx} F(s, X_s) X(s)(dx) ds \\
+ \ &\hphantom{\frac{1}{2}}  \int_0^t \int_E  \mathcal{D}_x F (s, X_s) M(dx, ds)
\end{align*}
holds.
\end{Thm}

\begin{proof}
Define the dyadic   mesh $\{ \tau_k^n, \ k = 1 , \ldots , k(n) \}$ on $[0,t]$ by
\begin{equation*}
\tau_0^n = 0, \quad \tau_k^n = \inf \{ s > \tau_{k-1}^n \ | \ 2^n s \in \mathbb{N} \} \wedge t.
\end{equation*}
Further, define 
\begin{equation*}
App^n(X_t) = \sum_{i = 0}^\infty X(\tau_{i+1}^n ) 1_{[\tau_i^n , \tau_{i+1}^n)} + X(t) 1_{[t , T]}.
\end{equation*}
While $X$ itself has continuous paths, this is a cadlag, piecewise constant approximation of $X_t$. Now set $h_i^n = \tau_{i+1}^n - \tau_i^n$. Then 
\begin{align}
& F(\tau_{i+1}^n, App^n(X_t)_{\tau_{i+1}^n -}) - F(\tau_i^n, App^n(X_t)_{\tau_i^n -})  \notag \\
= \quad & F (\tau_{i+1}^n , App^n(X_t)_{\tau_{i+1}^n -}) - F(\tau_i^n, App^n(X_t)_{\tau_i^n}) \label{sum}\\
+ & \ F(\tau_i^n, App^n(X_t)_{\tau_i^n}) - F (\tau_i^n, App^n(X_t)_{\tau_i^n -}) \notag
\end{align}

The first part of the right hand side in \eqref{sum} is equal to $\psi(h_i^n) - \psi(0)$ where $\psi(u) = F(\tau_i^n + u, App^n(X_t)_{\tau_i^n })$ as
\begin{equation*}
\psi(h_i^n) - \psi(0) = F(\tau_i^n + h_i^n, App^n(X_t)_{\tau_i^n}) - F(\tau_i^n,  App^n(X_t)_{\tau_i^n}) ,
\end{equation*}
and for all $u \in [0,T]$
\begin{align*}
App^n(X_t)_{\tau_{i+1}^n-}(u) &= \begin{cases} 
App^n(X_t)(u) \quad &u \in [0, \tau_{i+1}^n) \\
App^n(X_t)(\tau_{i+1}^n-) \quad &u \in [\tau_{i+1}^n, T]
\end{cases} \\
&= \begin{cases} 
App^n(X_t)(u) \quad &u \in [0, \tau_{i+1}^n) \\
App^n(X_t)(\tau_i^n) \quad &u \in [\tau_{i+1}^n, T]
\end{cases} \\
&= \begin{cases} 
App^n(X_t)(u) \quad &u \in [0, \tau_{i}^n) \\
App^n(X_t)(\tau_i^n) \quad &u \in [\tau_{i}^n, T]
\end{cases}  \\
&= App^n(X_t)_{\tau_i^n} (u) .
\end{align*}
Therefore, we have
\begin{align*}
&F(\tau_{i+1}^n, App^n(X_t)_{\tau_i^n}) - F (\tau_i^n,  App^n(X_t)_{\tau_i^n}) \\
= & \int_0^{\tau_{i+1}^n - \tau_i^n}  \mathcal{D}^*  F ({\tau_i^n + s}, App^n(X_t)_{\tau_i^n})  ds \\
= & \int_{\tau_i^n}^{\tau_{i+1}^n} \mathcal{D}^* F(s, App^n(X_t)_{\tau_i^n}) ds
\end{align*}
as $\psi(h_i^n) - \psi(0) = \int_0^{h_i^n} \psi'(u) du$ and
\begin{align*}
\psi'(u) &= \lim_{\varepsilon \rightarrow 0} \frac{\psi(u + \varepsilon) - \psi(u)}{\varepsilon} \\
&= \lim_{\varepsilon \rightarrow 0 } \frac{F (\tau_i^n + u + \varepsilon , App^n(X_t)_{\tau_i^n}) - F (\tau_i^n + u, App^n(X_t)_{\tau_i^n})}{\varepsilon} \\
&= \mathcal{D}^* F(\tau_i^n + u, App^n(X_t)_{\tau_i^n}) .
\end{align*}

The second term of the right hand side in \eqref{sum} is equal to $\phi(X(\tau_{i+1}^n) ) - \phi(X(\tau_i^n))$ where $\phi(\mu) = \tilde{\phi}(\mu - X(\tau_i^n))$, $\tilde{\phi}(\mu) = F(\tau_i^n, App^n(X_t)_{\tau_i^n -} + \mu 1_{[\tau_i^n , T]} )$ as
\begin{align*}
&\phi(X(\tau_{i+1}^n)) - \phi(X(\tau_i^n)) \\
= \ &F (\tau_i^n, App^n(X_t)_{\tau_i^n -} +  (X(\tau_{i+1}^n)- X(\tau_i^n) )1_{[\tau_i^n , T]}  ) - F (\tau_i^n, App^n(X_t)_{\tau_i^n -})
\end{align*}
and
\begin{align*}
&  App^n(X_t)_{\tau_i^n -} +  (X(\tau_{i+1}^n)- X(\tau_i^n) )1_{[\tau_i^n , T]} (u) \\
 = \ & \begin{cases}
App^n(X_t)(u)  \quad &\text{$u \in [0, \tau_i^n)$} \\
App^n(X_t)(\tau_i^n-) + X(\tau_{i+1}^n ) - X(\tau_i^n) \quad & \text{$u \in [\tau_i^n , T]$}
 \end{cases} \\
  = \ & \begin{cases}
App^n(X_t)(u)  \quad &\text{$u \in [0, \tau_i^n)$} \\
X(\tau_{i+1}^n ) \quad & \text{$u \in [\tau_i^n , T]$}
 \end{cases} \\
= \ & App^n(X_t)_{\tau_i^n} .
\end{align*}
As 
\begin{align*}
D_x \phi(\mu) &= \lim_{\varepsilon \rightarrow 0} \frac{\phi(\mu + \varepsilon \delta_x) -\phi(\mu)}{\varepsilon} \\
&= \lim_{\varepsilon \rightarrow 0} \frac{1}{\varepsilon} \left(F(\tau_i^n , App^n(X_t)_{\tau_i^n- } + (\mu - X(\tau_i^n) ) 1_{[\tau_i^n , T]} + \varepsilon \delta_x 1_{[\tau_i^n, T]}) \right. \\
& \hphantom{\lim_{\varepsilon \rightarrow 0} \frac{1}{\varepsilon} (} \quad - \left. F(\tau_i^n, App^n(X_t)_{\tau_i^n-} + (\mu - X(\tau_i^n)) 1_{[\tau_i^n , T]})\right) \\
&= \mathcal{D}_x F (\tau_i^n, App^n(X_t)_{\tau_i^n -} + (\mu  - X(\tau_i^n)) 1_{[\tau_i^n, T]} )
\end{align*}
and analogously 
\begin{align*}
D_{x_1x_2}\phi(\mu) &= \mathcal{D}_{x_1x_2} F (\tau_i^n, App^n(X_t)_{\tau_i^n -} + (\mu  - X(\tau_i^n)) 1_{[\tau_i^n, T]} ) \\
D_{x_1 x_2 x_3} \phi(\mu) & = \mathcal{D}_{x_1 x_2 x_3} F (\tau_i^n, App^n(X_t)_{\tau_i^n -} + (\mu  - X(\tau_i^n)) 1_{[\tau_i^n, T]} ),
\end{align*} 
we can apply the non-funcitonal It\={o}-formula from Theorem \ref{Thm.Ito} (see \ref{AppendA}) and obtain
\begin{align*}
\phi(X(\tau_{i+1}^n)) - \phi(X(\tau_i^n)) = \quad &\hphantom{\frac{1}{2}}\int_{\tau_i^n}^{\tau_{i+1}^n} \int_E A^{(x)} \mathcal{D}_x \phi(X(s)) X(s)(dx)ds \\
+ \ & \frac{1}{2} \int_{\tau_i^n}^{\tau_{i+1}^n} \int_E c \mathcal{D}_{xx} \phi(X(s)) X(s) (dx) ds\\
+ \ & \hphantom{\frac{1}{2}} \int_{\tau_i^n}^{\tau_{i+1}^n} \int_E \mathcal{D}_x \phi(X(s)) M(ds, dx) .
\end{align*}
Plugging in the definition of $\phi$, we end up with
\begin{align*}
&\hphantom{\frac{1}{2}} F (\tau_{i}^n, App^n(X_t)_{\tau_{i}^n}) - F (\tau_i^n,  App^n(X_t)_{\tau_i^n-})\\
 = \quad &\hphantom{\frac{1}{2}}\int_{\tau_i^n}^{\tau_{i+1}^n} \int_E A^{(x)} \mathcal{D}_x F (\tau_i^n, App^n(X_t)_{\tau_i^n -} + (X(s)  - X(\tau_i^n)) 1_{[\tau_i^n, T]} ) X(s)(dx)ds \\
+ \ & \frac{1}{2} \int_{\tau_i^n}^{\tau_{i+1}^n} \int_E c \mathcal{D}_{xx} F (\tau_i^n, App^n(X_t)_{\tau_i^n -} + (X(s)  - X(\tau_i^n)) 1_{[\tau_i^n, T]} ) X(s) (dx) ds\\
+ \ & \hphantom{\frac{1}{2}} \int_{\tau_i^n}^{\tau_{i+1}^n} \int_E \mathcal{D}_x F (\tau_i^n, App^n(X_t)_{\tau_i^n -} + (X(s)  - X(\tau_i^n)) 1_{[\tau_i^n, T]} ) M(ds, dx) .
\end{align*}

Combining this with the result for the first part of the sum in \eqref{sum} yields the following expression for the left hand side in \eqref{sum}:
\begin{align*}
&\hphantom{\frac{1}{2}} F (\tau_{i+1}^n, App^n(X_t)_{\tau_{i+1}^n-}) - F (\tau_i^n, App^n(X_t)_{\tau_i^n-})\\
 = \quad & \hphantom{\frac{1}{2}} \int_{\tau_i^n}^{\tau_{i+1}^n} \mathcal{D}^* F(s, App^n(X_t)_{\tau_i^n}) ds \\
 + \ & \hphantom{\frac{1}{2}}\int_{\tau_i^n}^{\tau_{i+1}^n} \int_E A^{(x)} \mathcal{D}_x F (\tau_i^n, App^n(X_t)_{\tau_i^n -} + (X(s)  - X(\tau_i^n)) 1_{[\tau_i^n, T]} ) X(s)(dx)ds \\
+ \ & \frac{1}{2} \int_{\tau_i^n}^{\tau_{i+1}^n} \int_E c \mathcal{D}_{xx} F (\tau_i^n, App^n(X_t)_{\tau_i^n -} + (X(s)  - X(\tau_i^n)) 1_{[\tau_i^n, T]} ) X(s) (dx) ds\\
+ \ & \hphantom{\frac{1}{2}}  \int_{\tau_i^n}^{\tau_{i+1}^n} \int_E \mathcal{D}_x F (\tau_i^n, App^n(X_t)_{\tau_i^n -} + (X(s)  - X(\tau_i^n)) 1_{[\tau_i^n, T]} ) M(ds, dx) .
\end{align*}

Defining $i_n(s)$ as the index such that $s \in [\tau_{i_n(s)}^n, \tau_{i_n(s) + 1}^n )$ and summation over $i$ yield:
\begin{align*}
&\hphantom{\frac{1}{2}}  F(t , App^n(X_t)_{t-}) - F (0, X_0)\\
 = \quad & \hphantom{\frac{1}{2}} \int_0^t \mathcal{D}^* F(s, App^n(X_t)_{\tau_{i_n(s)}^n}) ds \\
 + \ & \hphantom{\frac{1}{2}} \int_0^t \int_E A^{(x)} \mathcal{D}_x F (\tau_{i_n(s)}^n, App^n(X_t)_{\tau_{i_n(s)}^n -} + (X(s)  - X(\tau_{i_n(s)}^n)) 1_{[\tau_{i_n(s)}^n, T]} ) X(s)(dx) ds\\
+ \ & \frac{1}{2} \int_0^t \int_E c \mathcal{D}_{xx} F (\tau_{i_n(s)}^n, App^n(X_t)_{\tau_{i_n(s)}^n -} + (X(s)  - X(\tau_{i_n(s)}^n)) 1_{[\tau_{i_n(s)}^n, T]} ) X(s) (dx) ds\\
+ \ & \hphantom{\frac{1}{2}}  \int_0^t \int_E \mathcal{D}_xF (\tau_{i_n(s)}^n, App^n(X_t)_{\tau_{i_n(s)}^n -} + (X(s)  - X(\tau_{i_n(s)}^n)) 1_{[\tau_{i_n(s)}^n, T]} ) M(ds, dx) .
\end{align*}

Let's examine the limits of the different terms individually. First, we get
\begin{align*}
&d_\infty( (s, X_s), (\tau_{i_n(s)}^n, App^n(X_t)_{\tau_{i_n(s)}^n})) \\
 = \ & | s - \tau_{i_n(s)}^n| + \sup_{u \in [0,T]} d_P(X(s \wedge u), App^n(X_t)_{\tau_{i_n(s)}^n}(u)) \\
 \leq \ & \frac{1}{2^n} +  \sup_{0 \leq i \leq k(n)} \sup_{u \in [\tau_i^n , \tau_{i+1}^n)} d_P(X(u \wedge s), X(\tau_{i_n(s)}^n \wedge \tau_{i+1}^n)) ,
\end{align*}
which goes to zero due to the continuity of the paths of $X$. Analogously, we obtain that
$d_\infty((t, App^n(X_t)), (t, X_t))$ goes to zero. In particular, this yields $App^n(X_t)_{t-} \rightarrow X_{t-}$ and by combining this with the continuity of $F$ we end up with
\begin{equation*}
\lim_{n \rightarrow \infty} F (t, App^n(X_t)_{t-})) = F(t, X_{t-}) .
\end{equation*}
Taking advantage of the continuity of the paths of $X$ once again yields that $F(t, X_{t-}) = F(t, X_t)$. \\

As
\begin{equation*}
d_\infty ((s, X_s), (s, App^n(X_t)_{\tau_{i_n(s)}^n} )) \rightarrow 0 ,
\end{equation*}
the continuity assumptions on $\mathcal{D}^* F$ yield
\begin{equation*}
\lim_{n \rightarrow \infty}\mathcal{D}^* F(s, App^n(X_t)_{\tau_{i_n(s)}^n} )) = \mathcal{D}^*F(s, X_s) .
\end{equation*}
The boundedness assumptions on $\mathcal{D}^* F$ further allow us to apply the dominated convergence theorem to end up with
\begin{equation*}
\lim_{n \rightarrow \infty} \int_0^t \mathcal{D}^* F(s, App^n(X_t)_{\tau_{i_n(s)}^n}) ds = \int_0^t \mathcal{D}^* F(s, X_s)ds .
\end{equation*}

Now consider
\begin{align*}
&d_\infty((s, X_s)(\tau_{i_n(s)}^n , App^n(X_t)_{\tau_{i_n(s)}^n -} + (X(s) - X(\tau_{i_n(s)}^n ))1_{[\tau_{i_n(s)}^n, T]} )) \\
 = \ & | s - \tau_{i_n(s)}^n  | + \sup_{u \in [0,T]} d_P (X_s(u) , App^n(X_t)_{\tau_{i_n(s)}^n -}(u) + (X(s) - X(\tau_{i_n(s)}^n ))1_{[\tau_{i_n(s)}^n, T]}(u)).
\end{align*}
As $| s - \tau_{i_n(s)}^n | \rightarrow 0$, we only have to be concerned with
\begin{align*}
& \hphantom{+} \sup_{u \in [0,T]} d_P (X_s(u) , App^n(X_t)_{\tau_{i_n(s)}^n - }(u) + (X(s) - X(\tau_{i_n(s)}^n ))1_{[\tau_{i_n(s)}^n, T]}(u)) \\
\leq  \ & \hphantom{+} \sup_{u \in [0,T]} d_P (X_s(u),  App^n(X_t)_{\tau_{i_n(s)}^n -} (u) ) \\
& + \sup_{u \in [0,T]}  d_P (App^n(X_t)_{\tau_{i_n(s)}^n-} (u), App^n(X_t)_{\tau_{i_n(s)}^n - }(u) + (X(s) - X(\tau_{i_n(s)}^n ))1_{[\tau_{i_n(s)}^n, T]}(u) ) .
\end{align*}
The first term on the right hand side goes to zero, which leaves us with the second term. However, as
\begin{align*}
&\sup_{u \in [0,T]} d_P ( App^n(X_t)_{\tau_{i_n(s)}^n-} (u) , App^n(X_t)_{\tau_{i_n(s)}^n - }(u) + (X(s) - X(\tau_{i_n(s)}^n ))1_{[\tau_{i_n(s)}^n, T]}(u)) \\
= & \ d_P(X(\tau_{i_n(s)}^n),X(\tau_{i_n(s)}^n) + X(s) - X(\tau_{i_n(s)}^n) ) \\
= & \ d_P(X(\tau_{i_n(s)}^n),X(s)) ,
\end{align*}
the second term on the right hand side also goes to zero and thus the continuity assumptions yield
\begin{equation*}
\lim_{n \rightarrow \infty} A^{(x)} \mathcal{D}_x F (\tau_{i_n(s)}^n, App^n(X_t)_{\tau_{i_n(s)}^n - } + (X(s) - X(\tau_{i_n(s)}^n ))1_{[\tau_{i_n(s)}^n, T]}) = A^{(x)} \mathcal{D}_x F_s (X_s) .
\end{equation*}

Like in the proof of Theorem \ref{Thm.Ito}, it is enough to consider $X \in M_K(E)$. Then the boundedness assumptions on $A^{(x)}\mathcal{D}_x F_s(\omega)$ allow us to apply the dominated convergence theorem (see \ref{AppendB}). Therefore, we end up with
\begin{align*}
& \lim_{n \rightarrow \infty} \int_0^t \int_E A^{(x)} \mathcal{D}_x F (\tau_{i_n(s)}^n, App^n(X_t)_{\tau_{i_n(s)}^n - } + (X(s) - X(\tau_{i_n(s)}^n ))1_{[\tau_{i_n(s)}^n, T]}) X(s)(dx)ds \\
= \ & \int_0^t \int_E A^{x} \mathcal{D}_x F(s, X_s) X(s)(dx) ds .
\end{align*}
Analogously we obtain
\begin{align*}
& \lim_{n \rightarrow \infty} \int_0^t \int_E \mathcal{D}_{xx} F(\tau_{i_n(s)}^n,App^n(X_t)_{\tau_{i_n(s)}^n - } + (X(s) - X(\tau_{i_n(s)}^n ))1_{[\tau_{i_n(s)}^n, T]}) X(s) (dx) ds \\
= \ & \int_0^t \int_E \mathcal{D}_{xx} F (s, X_s) X(s)(dx)ds .
\end{align*}

For the convergence of the integral with respect to $M(ds, dx)$, note that 
\begin{equation*}
\mathcal{D}_x F (\tau_{i_n(s)}^n , App^n(X_t)_{\tau_{i_n(s)}^n - } + (X(s) - X(\tau_{i_n(s)}^n ))1_{[\tau_{i_n(s)}^n, T]})
\end{equation*}
is bounded and as the function is optional, it is in $\mathcal{P}_M$. The same applies to $\mathcal{D}_x F (s, X_s)$ and thus, the integrals are well defined. \\

From the continuity assumptions on $\mathcal{D}_x F(t, \omega)$ and the fact that
\begin{equation*}
d_\infty((s, X_s)(\tau_{i_n(s)}^n , App^n(X_t)_{\tau_{i_n(s)}^n -}) + (X(s) - X(\tau_{i_n(s)}^n ))1_{[\tau_{i_n(s)}^n, T]} )) \rightarrow 0 ,
\end{equation*}
we get that
\begin{equation}
\sup_{x \in E, \, s \in [0,T]} | \mathcal{D}_x F(\tau_{i_n(s)}^n,  \omega_n) - \mathcal{D}_x F(s,\omega) | \rightarrow 0 
\label{supconv}
\end{equation}
if $\omega_n \rightarrow \omega$. This has to hold as otherwise, we can find an $\varepsilon >0$ and sequences $(s_n) \subset [0,T]$ and $(x_n) \subset E$ such that
\begin{equation}
| \mathcal{D}_{x_n} F(\tau_{i_n(s)}^n, \omega_n) - \mathcal{D}_{x_n} F(s_n, \omega) | \geq 0 \quad \forall n \in \mathbb{N}.
\label{contra}
\end{equation}
However, there exist convergent subsequences $(\tilde{s}_n)$ and $(\tilde{x}_n)$ of $(s_n)$ and $(x_n)$ with $\tilde{s}_n \rightarrow \tilde{s}$ and $\tilde{x}_n \rightarrow \tilde{x}$ for some $\tilde{s}$, $\tilde{x}$ and the continuity assumptions as well as the convergence of $\omega_n$ yield
\begin{equation*}
| \mathcal{D}_{\tilde{x}_n} F(\tau_{i_n(\tilde{s}_n)}^n, \omega_n) - \mathcal{D}_{\tilde{x}} F(\tilde{s}, \omega)| \rightarrow 0
\end{equation*}
and
\begin{equation*}
| \mathcal{D}_{\tilde{x}_n} F(\tilde{s}_n, \omega) - \mathcal{D}_{\tilde{x}} F(\tilde{s}, \omega)| \rightarrow 0,
\end{equation*}
which contradicts \eqref{contra}. Therefore \eqref{supconv} has to hold. \\

From equation \eqref{supconv} we get that
\begin{align*}
\mathbb{E} \left[ c \int_0^t \int_E \right. |  \mathcal{D}_x F (\tau_{i_n(s)}^n , App^n(X_t)_{\tau_{i_n(s)}^n - } + (X(s) - X(\tau_{i_n(s)}^n ))1_{[\tau_{i_n(s)}^n, T]}) \quad \hphantom{\rightarrow \infty}  \\
  - \mathcal{D}_sF(s, X_s) |^2 X(s)(dx)ds \left. \vphantom{\int_0^T} \right] \rightarrow \infty .
\end{align*}
By the definition of the integral with respect to a martingale measure in \cite{Wal86}, we thus obtain
\begin{align*}
& \lim_{n \rightarrow \infty} \int_0^t \int_E \mathcal{D}_x F (\tau_{i_n(s)}^n , App^n(X_t)_{\tau_{i_n(s)}^n - } + (X(s) - X(\tau_{i_n(s)}^n ))1_{[\tau_{i_n(s)}^n, T]}) M(ds, dx) \\
= \ & \int_0^t \int_E \mathcal{D}_x F(s, X_s) M(ds, dx) ,
\end{align*}
which completes the proof.
\end{proof}

\section{Applications and further Notes} \label{SecApp}

We start this section by comparing our definition of the pathwise directional
derivatives to the derivatives obtained by following the approach in \cite{Du09}. This is followed by the derivation of martingale representation formulas from the two It\={o}-formulas in the previous sections (cf. \cite{EP95}).

\subsection{Comparison to Dupire's Vector Bundle Approach}

The approach in \cite{Du09} transfers to our setting as follows. Let $\tilde{\Lambda}_t = D([0, t] , M_F(E))$. The vector bundle underlying Dupire's idea is $\tilde{\Lambda} = \bigcup_{t \in [0, T]} \tilde{\Lambda}_t$. Further let $\tilde{\omega}_t = \{\omega(s)  :  0 \leq s \leq t\}$. Then $\tilde{\omega}_t \in \tilde{\Lambda}_t \subset \tilde{\Lambda}$. In \cite{Du09}, functional derivatives are defined as follows. Define, for some $\varepsilon \in \mathbb{R}$ and $x \in E$,
\begin{equation*}
\tilde{\omega}_t^{\varepsilon, x} = \tilde{\omega}_t(s) 1_{\{s \leq t \}} + \varepsilon \delta_x 1_{\{t = s \}}
\end{equation*}
and, for some functional $f : \tilde{\Lambda} \rightarrow \mathbb{R}$,
\begin{equation*}
\Delta_x f(\tilde{\omega}_t) = \lim_{\varepsilon \rightarrow 0} \frac{ f(\tilde{\omega}^{\varepsilon,x}_t) - f(\tilde{\omega}_t)}{\varepsilon} = \lim_{\varepsilon \rightarrow 0} \frac{f(\tilde{\omega}_t + \varepsilon \delta_x 1_{\{t \}} ) - f(\tilde{\omega}_t)}{\varepsilon} .
\end{equation*}

Consider the mapping $\varphi : \Lambda \rightarrow D([0, T], M_F(E))$ given by 
\begin{equation*}
\varphi (\tilde{\omega}_t) = \varphi (\{\omega(s) \ : \ 0 \leq s \leq t\} ) = \{ \omega(s \wedge t ) \ : \ 0 \leq s \leq T \}) = \omega_t .
\end{equation*}
Essentially, $\varphi$ extends the path of an element in $\Lambda_t$ up to time $T$ by keeping its final value $\omega(t)$ constant on $(t,T]$. Its inverse cuts the path of $\omega_t$ at time $t$. If $\omega$ was real-valued, the mapping could be illustrated as follows: 
\begin{align*}
\varphi( \includegraphics[width = 3 cm, height = 1 cm, valign = c]{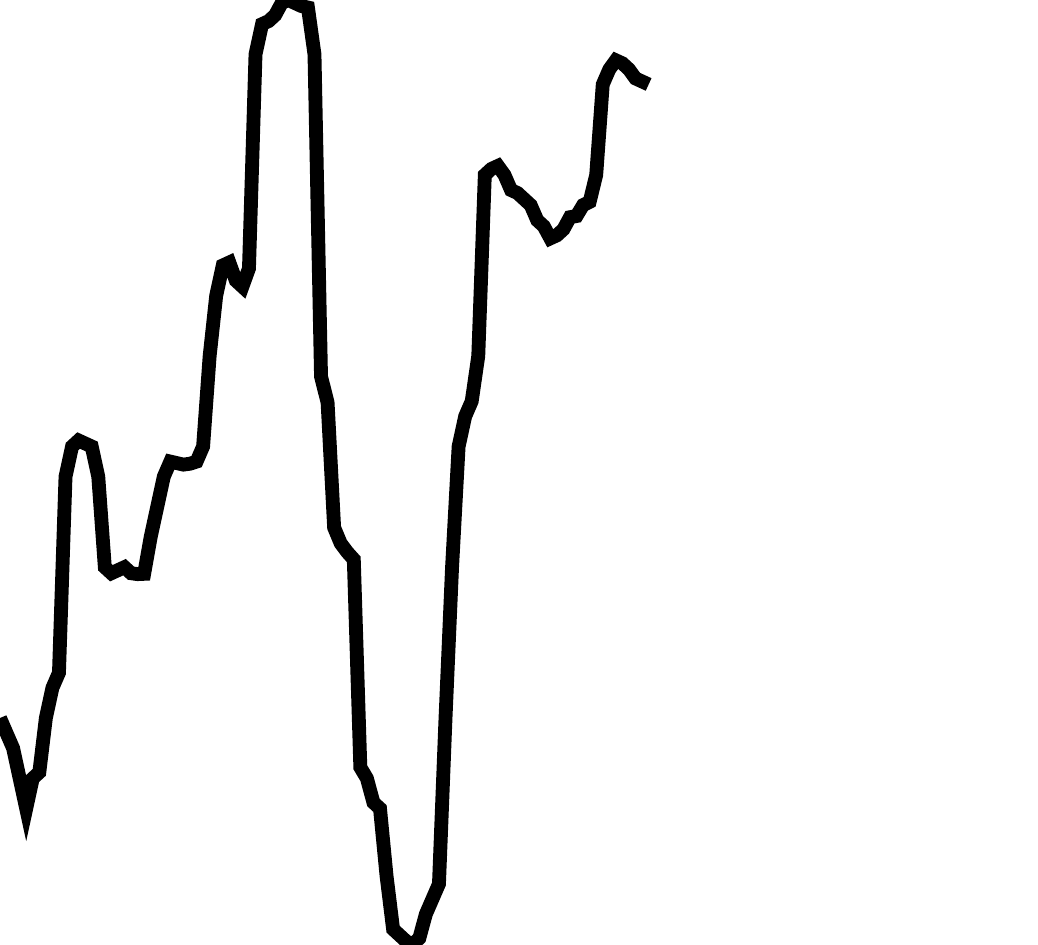} ) &= \includegraphics[width = 3 cm, height = 1 cm, valign = c]{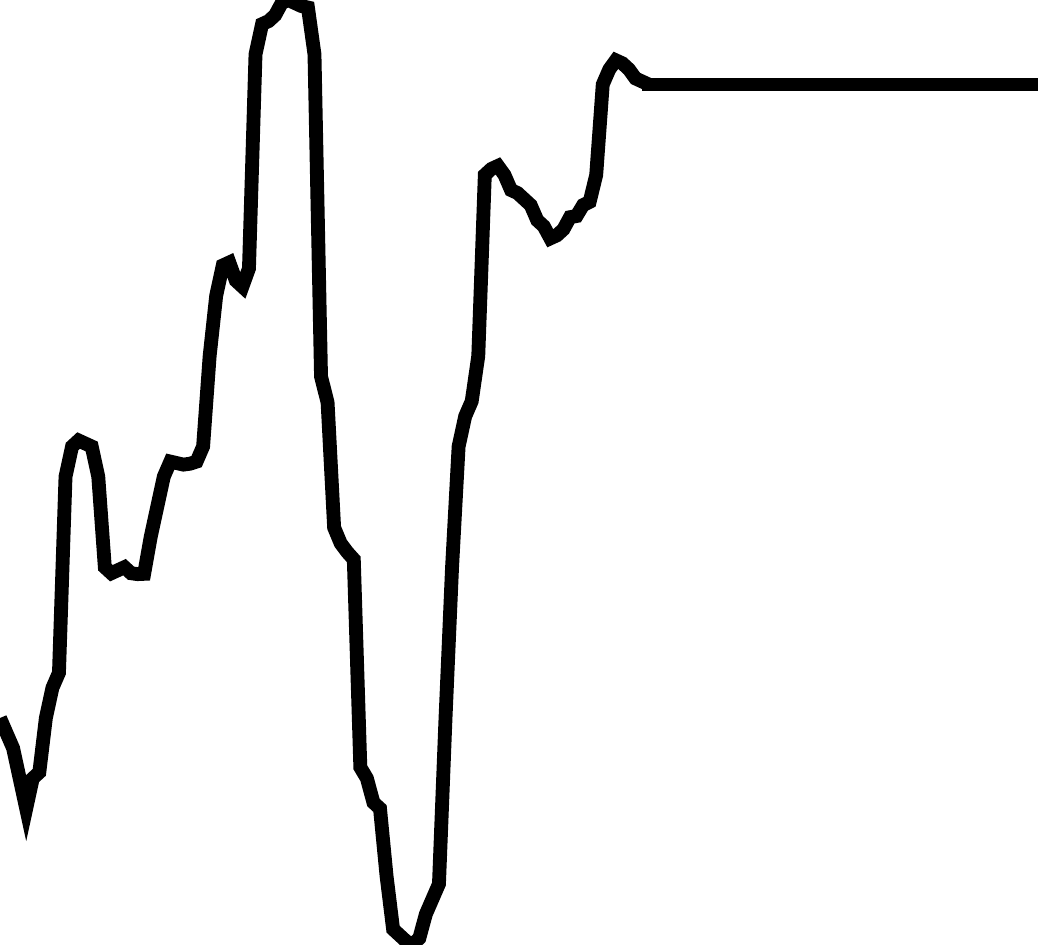}  \\
\varphi^{-1}( \includegraphics[width = 3 cm, height = 1 cm, valign = c]{Plot1.pdf} ) &= \includegraphics[width = 3 cm, height = 1 cm, valign = c]{Plot2.pdf}  
\end{align*}

Now, define $f = F \circ \varphi$ to connect our setting to the setting in \cite{Du09}. The function $f : \tilde{\Lambda} \rightarrow \mathbb{R}$ maps elements of the vector bundle onto the real line. As 
\begin{equation*}
\varphi ( \tilde{\omega}_t + \varepsilon \delta_x 1_{\{t\}}) = X_t + \varepsilon \delta_x 1_{[t,T]}
\end{equation*}
holds, we also get that the respective derivatives coincide:
\begin{align*}
\Delta_x f(\tilde{\omega}_t)  &= \lim_{\varepsilon \rightarrow 0} \frac{f(\tilde{\omega}_t + \varepsilon \delta_x 1_{\{s = t\}}) - f(\tilde{\omega}_t)}{\varepsilon} \\
&= \lim_{\varepsilon \rightarrow 0} \frac{F(\omega_t + \varepsilon \delta_x 1_{[ t, T]}) - F(\omega_t)}{\varepsilon} \\
&= \mathcal{D}_x F(\omega_t) .
\end{align*}
Higher order derivatives follow accordingly.

\subsection{Martingale Representation}

Theorem \ref{Thm.fct.Ito} allows us to derive the following martingale representation formula.

\begin{Thm}\label{ThmMR}
Let $F$ satisfy Conditions \ref{cond2}, $X$ solve \eqref{MP} and let $(F(t, X_t))_{t \in [0,T]}$ be a martingale. Then 
\begin{equation*}
F(t, X_t) = F(0, X_0) + \int_0^t \int_E \mathcal{D}_x F(s, X_s) M(ds, dx) , 
\end{equation*}
where $M$ is the martingale measure arising from \eqref{MP}. 
\end{Thm}

\begin{proof}
From Theorem \ref{Thm.fct.Ito} we get that
\begin{align*}
F(t, X_t) - F(0 , X(0)) \ = \quad &\hphantom{\frac{1}{2}}   \int_0^t \mathcal{D}^* F( s, X_s) ds \\
+ \ & \hphantom{\frac{1}{2}} \int_0^t \int_E A^{(x)} \mathcal{D}_x F(s, X_s) X(s)(dx) ds \\
+ \ & \frac{1}{2} \int_0^t \int_E c \mathcal{D}_{xx} F(s, X_s) X(s)(dx) ds \\
+ \ &\hphantom{\frac{1}{2}}  \int_0^t \int_E  \mathcal{D}_x F (s, X_s) M(dx, ds)
\end{align*}
However, since $F( t, X_t )$ is a martingale, the first three integrals of the sum vanish and thus
\begin{equation*}
F(t, X_t) = F(0, X_0) + \int_0^t \int_E \mathcal{D}_x F(s, X_s) M(ds, dx) ,
\end{equation*}
which is a martingale as $M$ is a martingale measure.
\end{proof}

\begin{Cor}
Let $F$ satisfy Conditions \ref{cond1}, $X$ solve \eqref{MP} and let $F(t, X(t))$ be a martingale. Then 
\begin{equation*}
F(t, X(t)) = F(0, X(0)) + \int_0^t \int_E  D_x F(s, X(s)) M(ds, dx) , 
\end{equation*}
where $M$ is the martingale measure arising from \eqref{MP}. 
\end{Cor}

\begin{proof}
Either repeat the proof of Theorem \ref{ThmMR} using Theorem \ref{Thm.Ito} or apply Theorem \ref{ThmMR} assuming $F(t, X(t)) = \tilde{F}(t, X_t)$ for a suitable $\tilde{F}$.
\end{proof}

\appendix

\renewcommand{\thesection}{Appendix \Alph{section}}

\section{Applying Theorem \ref{Thm.Ito} in the Proof of Theorem \ref{Thm.fct.Ito}} \label{AppendA}

In order to apply Theorem \ref{Thm.Ito} in the proof of Theorem \ref{Thm.fct.Ito} we have to confirm that $\phi$ satisfies Conditions \ref{cond1}. \\

Consider the following:
\begin{align*}
& App^n(X_t)_{\tau_i^n -} (u) + (\mu -X(\tau_i^n)) 1_{[\tau_i^n ,T]}(u) \\
= & \begin{cases}
App^n(X_t) (u) \quad &\text{ if $u \in [0, \tau_i^n)$} \\
App^n(X_t)(\tau_i^n -) + \mu - X(\tau_i^n) \quad &\text{ if $u \in  [\tau_i^n, T]$}
\end{cases} \\
= & \begin{cases}
App^n(X_t)_{\tau_i^n -} (u) \quad &\text{ if $u \in [0, \tau_i^n)$} \\
\mu  \quad &\text{ if $u \in [\tau_i^n, T]$}
\end{cases} .
\end{align*}
Now, as 
\begin{align*}
& d_\infty((\tau_i^n , App^n(X_t)_{\tau_i^n -}  + (\mu -X(\tau_i^n)) 1_{[\tau_i^n ,T]}), \\
& \hphantom{d_\infty(} (\tau_i^n,  App^n(X_t)_{\tau_i^n -}  + (\mu_m -X(\tau_i^n)) 1_{[\tau_i^n ,T]}) )\\
 = \ & \sup_{u \in [0,T]} d_P(App^n(X_t)_{\tau_i^n -} (u) + (\mu -X(\tau_i^n)) 1_{[\tau_i^n ,T]}(u), \\
 &\hphantom{\sup_{u \in [0,T]}d_P(}App^n(X_t)_{\tau_i^n -} (u) + (\mu_m -X(\tau_i^n)) 1_{[\tau_i^n ,T]}(u) ) \\
 = \ & d_P(\mu_m , \mu), 
\end{align*}
we get the continuity of $\phi$ with respect to $\mu$ from
\begin{align*}
&\{ \mu_m \xrightarrow{ m \rightarrow \infty} \mu\} \\
 \Rightarrow \quad &\{ d_P(\mu_m, \mu) \xrightarrow{ m \rightarrow \infty} 0\} \\
 \Rightarrow \quad &\{ d_\infty((\tau_i^n , App^n(X_t)_{\tau_i^n -}  + (\mu -X(\tau_i^n)) 1_{[\tau_i^n ,T]}), \\
& \hphantom{\{d_\infty( } \ (\tau_i^n,  App^n(X_t)_{\tau_i^n -}  + (\mu_m -X(\tau_i^n)) 1_{[\tau_i^n ,T]}) ) \xrightarrow{ m \rightarrow \infty} 0 \} \\
 \Rightarrow \quad & \{ |F(\tau_i^n, App^n(X_t)_{\tau_i^n -}^{\mu - X(\tau_i^n)}) - F(\tau_i^n, App^n(X_t)_{\tau_i^n -}^{\mu_n - X(\tau_i^n)})| \xrightarrow{ m \rightarrow \infty} 0 \}
\end{align*}
where the last part follows from the continuity of $F$. \\

As $\phi$ is independent of $s$, we do not have to consider the conditions with respect to $s$. Thus, for $\phi$ to fulfill part 1 of Conditions \ref{cond1}, we only have to consider the vertical derivatives of $\phi$. However, as 
\begin{align*}
D_x \phi(\mu) &= \mathcal{D}_x F (\tau_i^n, App^n(X_t)_{\tau_i^n - } + (\mu - X(\tau_i^n)) 1_{[\tau_i^n , T]} ) , \\
D_{x_1 x_2} \phi(\mu) &= \mathcal{D}_{x_1 x_2} F (\tau_i^n, App^n(X_t)_{\tau_i^n - } + (\mu - X(\tau_i^n)) 1_{[\tau_i^n , T]} ), \\
\end{align*}
and 
\begin{equation*}
D_{x_1 x_2 x_3} \phi(\mu) = \mathcal{D}_{x_1 x_2 x_3} F (\tau_i^n, App^n(X_t)_{\tau_i^n - } + (\mu - X(\tau_i^n)) 1_{[\tau_i^n , T]} ) ,
\end{equation*}
we get the continuity with respect to $x_1$, $x_2$, $x_3$ and $\mu$ from the conditions on $F$.  These requirements on $F$ and the same arguments as above also guarantee that the second and third set of conditions in Conditions \ref{cond1} are fulfilled by $\phi$, which allows us to apply Theorem \ref{Thm.Ito} to $\phi$.

\section{Applying the Dominated Convergence Theorem in the Proof of Theorem \ref{Thm.fct.Ito}.} \label{AppendB}

To simplify notations, set $\alpha^n(s) = \tau_{i_n(s)}^n$ and 
\begin{equation*}
\beta^n(s) = App^n(X_t)_{\tau_{i_n(s)}^n-} + (X(s) - X(\tau_{i_n(s)}^n )) 1_{[\tau_{i_n(s)}^n,T]} ).
\end{equation*}
Further, as $X$ is localized by assumption, $X$ is a bounded measure, i.e. there exists a constant $C > 0$ such that $X(s)(E) \leq C < \infty$ for all $s \in [0,T]$. \\

As we know that $d_\infty( (s, X_s), (\alpha^n(s), \beta^n(s))) \rightarrow 0$ as $n \rightarrow \infty$, the continuity of $A^{(x)}\mathcal{D}_x F(s, \omega)$ yields
\begin{equation*}
\lim_{n \rightarrow \infty} A^{(x)} \mathcal{D}_x F(\alpha^n(s), \beta^n(s)) = A^{(x)} \mathcal{D}_x F(s, X_s) .
\end{equation*}
We also have that $A^{(x)}\mathcal{D}_x F(s, \omega)$ is bounded and thus there exists an upper bound $B \in \mathbb{R}$ such that $\sup_{s \in [0,T], \, x \in E} | A^{(x)}\mathcal{D}_x F(\alpha^n(s), \beta^n(s)) | \leq B$ for all $n$ and
\begin{equation*}
\int_E | B | X(s)(dx)  = B X(s)(E) \leq B C  < \infty
\end{equation*}
for all $s \in [0,T]$. Thus we can apply the dominated convergence theorem to get
\begin{equation*}
\lim_{n \rightarrow \infty} \int_E A^{(x)} \mathcal{D}_x F(\alpha^n(s), \beta^n(s)) X(s)(dx) = \int_E A^{(x)} \mathcal{D}_x F(s, X_s) X(s)(dx)
\end{equation*}
for all $s \in [0,T]$. \\

Combining this with the fact that
\begin{align*}
| \int_E A^{(x)} \mathcal{D}_x F(\alpha^n(s), \beta^n(s)) X(s)(dx) | &\leq \int_E | A^{(x)} \mathcal{D}_x F(\alpha^n(s), \beta^n(s)) | X(s)(dx) \\
& \leq B C
\end{align*}
for all $n$ and
\begin{equation*}
\int_0^t | BC | ds = BC t < \infty,
\end{equation*}
we can once again apply the dominated convergence theorem to finally end up with
\begin{align*}
&\lim_{n \rightarrow \infty} \int_0^t \int_E A^{(x)} \mathcal{D}_x F(\alpha^n(s), \beta^n(s)) X(s)(dx)ds \\
= &\int_0^t \int_E A^{(x)} \mathcal{D}_x F(s, X_s) X(s)(dx) ds . 
\end{align*}

\section*{Acknowledgements}

The first author was supported by the Deutsche Forschungsgemeinschaft.

\bibliography{Preprint_SP}
\bibliographystyle{abbrv}

\end{document}